\documentclass{amsart}

\usepackage[english]{babel}

\usepackage{mathrsfs, mathtools, amssymb}

\usepackage{dsfont}

\usepackage{tikz, tikz-cd}
\usepackage[paper=a4paper, margin=2cm]{geometry}

\usepackage{hyperref}
\hypersetup{
    unicode=false,          % non-Latin characters in bookmarks
    pdftoolbar=true,        % show toolbar?
    pdfmenubar=true,        % show menu?
    pdffitwindow=false,     % window fit to page when opened
    pdfstartview={FitH},    % fits the width of the page to the window
    pdftitle={Cohomology Rings of Quasitoric bundles},    % title
    pdfauthor={Khovanskii, A. G.; Limonchenko, I.; Monin, L.},     % author
    pdfkeywords={Quasitoric bundles} {Quasitoric manifolds} {moment-angle-complexes}  {Stanley-Raisner rings}, % list of keywords
    pdfnewwindow=true,      % links in new window
    colorlinks=true,       % false: boxed links; true: colored links
    linkcolor=blue,          % color of internal links
    citecolor=red,        % color of links to bibliography
    filecolor=magenta,      % color of file links
    urlcolor=cyan           % color of external links
}

\theoremstyle{plain}
\newtheorem{theorem}{Theorem}[section]
\newtheorem{lemma}[theorem]{Lemma}
\newtheorem{proposition}[theorem]{Proposition}
\newtheorem{corollary}[theorem]{Corollary}

\theoremstyle{definition}
\newtheorem{definition}[theorem]{Definition}
\newtheorem{remark}[theorem]{Remark}
\newtheorem{example}[theorem]{Example}

\newcommand{\rleft}{\mathopen{}\mathclose\bgroup\left}
\newcommand{\rright}{\aftergroup\egroup\right}

\newcommand{\C}{{\mathbb{C}}}
\newcommand{\K}{{\mathbb{K}}}
\newcommand{\R}{{\mathbb{R}}}
\newcommand{\Z}{{\mathbb{Z}}}
\newcommand{\p}{{\mathbb{P}}}
\newcommand{\Q}{{\mathbb{Q}}}

\newcommand{\Pm}{{\mathcal{P}}}

\newcommand{\cL}{{\mathcal{L}}}

\newcommand{\Exp}{{\mathrm{Exp}}}
\newcommand{\Ann}{{\mathrm{Ann}}}
\newcommand{\BKK}{BKK }
\newcommand{\diff}{\mathop{}\!d}
\newcommand{\Diff}{{\mathrm{Diff}}}

\newcommand{\Hom}{{\mathrm{Hom}}}
\newcommand{\lspan}{{\operatorname{span}}}

\newcommand{\PP}{\mathrm{PP}}
\newcommand{\pt}{{\mathrm{pt}}}

\newcommand{\Sym}{{\mathrm{Sym}}}
\newcommand{\Vol}{{\mathrm{Vol}}}
\newcommand{\s}{{\mathrm{sign}}}

\begin{document}
\selectlanguage{english}

%%%%%%%%%%%%%%%%%%%%%%%%%%%%%%%%%%%%%%%%%%%%%%%%%%%%%%%%%%%%%%%% 

\title[Cohomology rings of quasitoric bundles]{Cohomology rings of quasitoric bundles}

\author{Askold Khovanskii}
\address[A.\,Khovanskii]{Department of Mathematics, University of Toronto, Toronto, Canada; Moscow Independent University, Moscow, Russia.}
\email{askold@math.utoronto.ca}
\author{Ivan Limonchenko}
\address[I.\,Limonchenko]{Department of Mathematics, University of Toronto, Toronto, Canada; Faculty of Computer Science, HSE University, Moscow, Russia}
\email{ilimonchenko@hse.ru}
\author{Leonid Monin}
\address[L.\,Monin]{Max Planck Institute for Mathematics in the Sciences, Leipzig, Germany}
\email{leonid.monin@mis.mpg.de}


\begin{abstract}
  The classical \BKK theorem computes the intersection number of divisors on toric variety in terms of volumes of corresponding polytopes. In \cite{KP}, it was observed by Pukhlikov and the first author that the \BKK theorem leads to a presentation of the cohomology ring of toric variety as a quotient of the ring of differential operators with constant coefficients by the annihilator of an explicit polynomial.
  
  In this paper we generalize this construction to the case of quasitoric bundles. These are fiber bundles with generalized quasitoric manifolds as fibers. First we obtain a generalization of the \BKK theorem to this case. Then we use recently obtained descriptions of the graded-commutative algebras which satisfy Poincar\'e duality to give a description of cohomology rings of quasitoric bundles.
\end{abstract}

\maketitle

%%%%%%%%%%%%%%%%%%%%%%%%%%%%%%%%%%%%%%%%%%%%%%%%%%%%%%%%%%%%%%%%%

\section{Introduction}
\label{sec:intro}

In this paper we study cohomology rings of quasitoric bundles, a class of fiber bundles over a closed orientable smooth manifold with a generalized quasitoric manifold as a fiber. A generalized quasitoric manifold is a generalization of a toric variety to smooth category, hence the first examples of a quasitoric bundle are a nonsingular complete toric variety and a quasitoric manifold, both viewed as fiber bundles over a point. 

{\bf Description of cohomology rings of a toric variety.}
Let us start with a brief recollection of the classical toric case. The cohomology ring of a smooth toric variety has several descriptions. In this paper we mostly focus on the following three descriptions: Stanley-Reisner, Brion's and Pukhlikov-Khovanskii.

Let $T\simeq (\C^*)^n$ be an algebraic torus with character lattice $M$ and lattice of one-parameter subgroups $N$.
Further, let $X_\Sigma$ be a smooth projective toric variety given by a fan $\Sigma \subseteq N_\R\coloneqq N \otimes_\Z \R$.
Denote the rays of $\Sigma$ by $\Sigma(1) = \{ \rho_1,\ldots,\rho_s\}$ and their primitive generators in $N$ by $e_1,\ldots, e_s$.
Recall the following well-known description of the cohomology ring of $X_\Sigma$ (see, for instance, \cite[Theorem 12.4.1]{tor-var}):
\[
  H^*(X_\Sigma, \R)\simeq \R[x_1,\ldots, x_s]/ (I + J) \eqqcolon R_\Sigma,
\]
where $I$ is generated by monomials $x_{i_1}\cdots x_{i_t}$ such that $\rho_{i_1},\ldots, \rho_{i_t}\in \Sigma(1)$ are distinct and do not form a cone in $\Sigma$ and $J = \rleft\langle \sum_{i=1}^s \chi(e_i)x_i \colon \chi \in M \rright\rangle$.
Note that $I$ depicts the Stanley-Reisner ideal of the fan $\Sigma$ and therefore we refer to this description as the \emph{Stanley-Reisner description} of $H^*(X_\Sigma,\R)$.

On the other hand, Brion used explicit localization theorem in equivariant cohomology to compute the cohomology ring of a smooth toric variety. Let $M$ be the character lattice of the torus $T$ and set $M_\R = M \otimes \R$ and $N_\R = \Hom_\Z(M, \R)$.
Let $\Sigma$ be a smooth projective fan in $N_\R$.
A map $f \colon N_\R \to \R$ is \emph{piecewise polynomial} if for any $\sigma \in \Sigma$, the map $f|_\sigma \colon \sigma \to \R$ extends to a polynomial function on the linear space $\lspan_\R\{ \sigma\}$, i.e. a piecewise polynomial function $f$ on $\Sigma$ is a collection of compatible polynomial functions $f_\sigma \colon \sigma \to \R$. 

Let us denote by $\PP_\Sigma$ the ring of piecewise polynomial functions on $\Sigma$ with respect to pointwise addition and multiplication. Note that the character lattice $M$ naturally injects in $\PP_\Sigma$ as global linear functions on $N_\R$. Brion shows that one has the following description of the cohomology ring of toric variety:
\[
 H^*(X_\Sigma, \R)\simeq \PP_\Sigma/\langle M\rangle,
\]
where $\langle M\rangle$ is the ideal generated by $M\subset \PP_\Sigma$.

Given line bundles $L_1, \ldots, L_t$ on $X_\Sigma$, one can directly compute a top degree cup product $c_1(L_1)^{k_1} \cdots c_1(L_t)^{k_t}$ in $H^*(X_\Sigma,\R)$ by using the \BKK theorem \cite{Kouchnirenko} (see also \cite{Bernstein, BKK}).
Here, $c_1(L_i) \in H^2(X_\Sigma,\R)$ denotes the first Chern class of $L_i$.
More precisely, as any line bundle on $X_\Sigma$ is the difference of two ample line bundles, it suffices to evaluate products $c_1(L_1)^{k_1} \cdots c_1(L_t)^{k_t}$ with all $L_i$ ample.
As is well-known in toric geometry, the ample line bundles $L_i$ correspond to polytopes $P_i$ whose normal fans coarsen the fan $\Sigma$.
Using the \BKK Theorem, we obtain
\[
  c_1(L_1)^{k_1} \cdots c_1(L_t)^{k_t} = n! \cdot V(\underbrace{P_1,\ldots, P_1}_{k_1\;\text{times}}, \ldots, \underbrace{P_t, \ldots, P_t}_{k_t\;\text{times}})
\]
where $V(P_1, \ldots, P_1, \ldots, P_t, \ldots, P_t)$ denotes the mixed volume of the $n$-tuple $(P_1, \ldots, P_1, \ldots, P_t, \ldots, P_t)$.
In \cite[Section 1.4]{KP}, Pukhlikov and the first author observed that the information on these cup products suffices to regain a description of the cohomology ring $H^*(X_\Sigma, \R)$. More precisely, they showed that
\[
H^*(X_\Sigma, \R)\simeq \Diff(\Pm_\Sigma)/\Ann(\Vol),
\]
where $\Diff(\Pm_\Sigma)$ is the ring of differential operators with constant coefficients on the space of virtual polytopes associated with $\Sigma$ and $\Ann(\Vol)$ is the ideal of differential operators which annihilate the volume polynomial.
As virtual polytopes play a crucial role in this description, we refer to it as the \emph{virtual polytope description} of the cohomology ring $H^*(X_\Sigma, \R)$.

{\bf Cohomology rings of quasitoric bundles.} Now let $T\simeq (S^1)^n$ be a compact torus. A quasitoric manifold is a smooth $2n$-dimensional manifold with a smooth effective locally standard $T$-action such that the orbit space is diffeomorphic to a simple polytope $P$. This notion was introduced as a topological counterpart of a nonsingular projective toric variety (or, toric manifold) with the restricted action of $(S^1)^n\subset (\C^*)^n$ in the seminal paper~\cite{davis1991convex}. Quasitoric manifolds have been studied intensively since 1990s and have already found various applications in homotopy theory~\cite{Choi2008QuasitoricMO,hasui2016p,hasui2017p}, unitary~\cite{Buchstaber2006SpacesOP,Lu2016EXAMPLESOQ} and special unitary bordism~\cite{lu2016toric,Limonchenko2017CalabiYH}, hyperbolic geometry~\cite{Buchstaber2016OnMD,Buchstaber2017CohomologicalRO,Baralic2020ToricOA}, and other areas of research. 

Quasitoric manifolds acquire smooth structures as orbit spaces of moment-angle manifolds by freely acting tori of maximal possible rank~\cite{buchstaber2002torus,Buchstaber2006SpacesOP}. In Section 2, we define a generalized quasitoric manifold to be a partial quotient obtained as an orbit space of a moment-angle-complex over a starshaped sphere by a freely acting torus of the maximal possible rank. It is a locally standard torus manifold~\cite{Masuda1999UnitaryTM,multifan,panovmasuda} with a poset of characteristic submanifolds being isomorphic to a complete simplicial fan.

Similar to quasitoric manifolds, generalized quasitoric manifolds have a combinatorial description: an omnioriented generalized quasitoric manifold is uniquely determined by a pair $(\Sigma,\Lambda)$, where $\Sigma$ is a complete simplicial fan and $\Lambda\colon \Sigma(1)\to N$ is a map such that $\Lambda(\rho_1),\ldots,\Lambda(\rho_l)$ can be extended to a basis of $N$, whenever $\rho_1,\ldots,\rho_l\in \Sigma(1)$ form a cone. 

In particular, for a smooth complete fan $\Sigma\subset N$ and $\Lambda(\rho)=e_\rho$, the primitive ray generator, the corresponding generalized quasitoric manifold is equivariantly diffeomorphic to a nonsingular complete toric variety $X_\Sigma$ (considered with the restricted action of $(S^1)^n\subset (\C^*)^n$).

For a principal $T$-bundle $p \colon E \to B$ over a closed orientable real manifold $B$ and a generalized quasitoric manifold $X_{\Sigma,\Lambda}$, let $E_{\Sigma,\Lambda} \coloneqq (E\times X_{\Sigma,\Lambda} / T)$ be the associated quasitoric bundle (see Section~\ref{sec:general} for details).

A particular example of quasitoric bundles is formed by toric bundles, where we consider associated bundles $E_\Sigma \coloneqq (E\times X_{\Sigma} / T)$ coming from smooth complete toric varieties. A generalization of the Stanley-Reisner type description for the cohomology ring $H^*(E_\Sigma,\R)$ was obtained by Sankaran and Uma \cite{US03}. In \cite{roch}, it was noticed that Brion's description also generalizes to the case of toric bundles. Finally, in \cite{hof2020} a generalization of the virtual polytope description was obtained for toric bundles with smooth projective fibers. Descriptions of cohomology rings obtained in \cite{roch} and \cite{hof2020} showed to be useful for computation of the ring of conditions of horospherical homogeneous spaces. Such a computation involves taking the direct limit of cohomology rings of toric bundles with the same base, which is descriptions of \cite{roch,hof2020} are more suitable.

On the other side, in \cite{davis1991convex} the Stanley-Reisner description was extended to the cohomology rings of quasitoric manifolds and in \cite{dasgupta2019cohomology} it was further extended to the case of quasitoric bundles. Moreover, in~\cite{ayzenberg2016volume} the cohomology rings of quasitoric manifolds were described via volume polynomial on the space of analogous  multi-polytopes. 

In this paper, we complete this picture and obtain a generalization of all the three descriptions of the cohomology ring of a smooth toric variety to the case of quasitoric bundles. In particular, we extend the results of \cite{hof2020} to the case of toric bundles with complete (but not necessarily projective) fibers and give an independent proof of the results of \cite{dasgupta2019cohomology} in the quasitoric bundle case.

Our main tool in the computation of cohomology rings is a topological version of \BKK theorem for quasitoric bundles which computes top-degree intersection numbers of cohomology classes of $E_{\Sigma,\Lambda}$ (Theorem~\ref{BKK}).

More precisely, we reduce the computation of intersection numbers on the quasitoric bundle to intersection numbers on the base.
This provides a ``\BKK-type theorem'' for any choice of a cohomology class in the base $\gamma \in H^*(B, \R)$.

We also provide an equivalent version of the generalized \BKK Theorem in Section 3. It can be made as follows.
Suppose the torus $T$ has rank $n$ and the real dimension of $B$ is $k$.
Similar to the toric case, a multi-polytope $\Delta$ defines a cohomology class $\rho(\Delta)\in H^2(E_{\Sigma,\Lambda},\R)$ on the quasitoric bundle $p \colon E_{\Sigma,\Lambda}\to B$. For any given $j$, we define a map which associates to a multi-polytope $\Delta$ a cohomology class $\gamma_{2j}(\Delta)\in H^{2j}(B,\R)$ such that
\[
  \rho(\Delta)^{n+j} \cdot p^*(\gamma) = \gamma_{2j}(\Delta) \cdot \gamma \text{,}
\]
for any $\gamma\in H^{k-2j}(B,\R)$.
Here ``$\cdot$'' denotes the cup product on the respective cohomology ring.
We call the class $\gamma_{2j}(\Delta)$ \emph{the horizontal part} of $\rho(\Delta)^{n+j}$. In Theorem~\ref{BKK1}, we compute the horizontal part of $\rho(\Delta)^{n+j}$ for any multi-polytope $\Delta$.

\begin{remark}
In this paper we deal with a topological version of BKK theorem. In particular, quasitoric bundles are locally trivial fiber bundles in smooth category over a smooth compact base. Thus, Theorems~\ref{BKK} and~\ref{BKK1} are natural extensions of the topological BKK theorems for toric bundles obtained in \cite{hof2020}. 

On the other hand, if one works only with algebraic toric bundles, one can use the theory of Newton-Okounkov bodies to relax assumptions on the base to be smooth or compact.  We will address this direction in the future work. 
\end{remark}

Our version of the \BKK theorem allows us to give a Pukhlikov-Khovanskii type description of the cohomology ring of quasitoric bundles. To do so we use the description of graded commutative algebras which satisfy Poincar\'e duality obtained in \cite{hof2020}. We then use this description to further prove a version of Stanley-Reisner and Brion descriptions of the cohomology ring of a quasitoric bundle in Theorems~\ref{thm:SR} and~\ref{thm:Br} respectively.

In the case when the base $B$ is an even-dimensional manifold (e.g. a quasitoric manifold), the subring 
\[
A = \bigoplus H^{2i}(E_{\Sigma,\Lambda},\R)
\]
of even-dimensional cohomology classes is a commutative Gorenstein ring. Using our version of the \BKK theorem stated in Section 3 and proved in Section 4 and the exact computation of a Macaulay's inverse system from~\cite{KhoM} and recalled in Section 5, we obtain an alternative description of the ring $A$ (see Theorem~\ref{subringevendim}) in Section 6.

\subsection*{Notations}
We will denote by $\cdot$ the cup product in cohomology. For a manifold $B$ and a cohomology class $\gamma \in H^*(B,\R)$, we will denote by $\langle\gamma,[B]\rangle$ the pairing of $\gamma$ with the fundamental class of $B$.
%%%%%%%%%%%%%%%%%%%%%%%%%%%%%%%%%%%%%%%%%%%%%%%%%%%%%%%%%%%%%%%% 

\subsection*{Acknowledgements}
The authors thank the Fields Institute for Research in Mathematical Sciences for the opportunity to work on this research project during the Thematic Program on Toric Topology and Polyhedral Products. We are grateful to Anton Ayzenberg, Victor Buchstaber, Johannes Hofscheier and Taras Panov for several helpful and inspiring conversations. 

The first author is partially supported by the Canadian Grant No.~156833-17. The research of the second author has been funded within the framework of the HSE University Basic Research Program and is also supported in part by Young Russian Mathematics award.

%%%%%%%%%%%%%%%%%%%%%%%%%%%%%%%%%%%%%%%%%%%%%%%%%%%%%%%%%%%%%%%%%
%%%%%%%%%%%% SECTION 2 %%%%%%%%%%%%%%%%%%%%%%%%%%%%%%%%%%%%%%%%%%
%%%%%%%%%%%%%%%%%%%%%%%%%%%%%%%%%%%%%%%%%%%%%%%%%%%%%%%%%%%%%%%%%

\section{Generalized quasitoric manifolds}
\label{sec:cohomology_toric}

In this section we introduce generalized quasitoric manifolds and discuss various approaches to describing their cohomology. 

\subsection{Definition}

We are going to define a notion, which is slightly more general than that of a quasitoric manifold; the methods we develop in this paper will work in this generality.

Let $K$ be an abstract simplicial complex of dimension $n-1$ on the vertex set $[s]=\{1,2,\ldots,s\}$. Recall that its moment-angle-complex $\mathcal Z_K$ is defined to be the $(s+n)$-dimensional cellular subspace in the unitary polydisc $(D^2)^s\subset\C^s$ given by the formula 
$\bigcup_{I\in K}\prod\limits_{i=1}^{s}{Y_i}$, where $Y_i=D^2$, if $i\in I$ and $Y_i=S^1$, otherwise. 

There is a natural (coordinatewise) action of the compact torus $T^s$ on $\mathcal Z_K$ and the orbit space $\mathcal Z_K/T^s$ is homeomorphic to the cone over the barycentric subdivision of $K$. When $K=\partial P^*$ is a polytopal sphere, the moment-angle-complex $\mathcal Z_K$ is homeomorphic to the moment-angle manifold $\mathcal Z_P$ and therefore acquires a smooth structure.

Let $\Sigma$ be a complete simplicial fan in $\R^n$ with the set of rays (1-dimensional generators) $\Sigma(1)=\{\rho_1,\ldots,\rho_s\}$. The intersection $K_\Sigma$ of the fan $\Sigma$ with the unit sphere $S^{n-1}\subset\R^n$ is, by definition, a starshaped triangulated sphere. Note that already in dimension 3, there exist non-polytopal starshaped spheres. Panov and Ustinovsky~\cite{PanUs} proved that $\mathcal Z_K$ has a smooth structure if and only if $K$ is a starshaped sphere.

Throughout the paper we denote by $M$ and $N$ the character and cocharacter lattices of $T^n$, respectively. 
Suppose $K=K_\Sigma$ is a starshaped sphere and $\Lambda: \Sigma(1) \to N$ is a \emph{characteristic map}, i.e. such a map that the collection of vectors
\[
\Lambda(\rho_1),\ldots, \Lambda(\rho_k) 
\] 
can be completed to a basis of the cocharacter lattice $N$, whenever $\rho_1,\ldots,\rho_k$ generate a cone in $\Sigma$.

Then the $(s-n)$-dimensional subtorus $H:=\ker\exp\Lambda\subset T^s$ acts freely on $\mathcal Z_K$ and the smooth manifold $X_{\Sigma,\Lambda}:=\mathcal Z_K/H$ will be called a \emph{generalized quasitoric manifold}. 

\begin{remark}
The maximal rank $s(K)$ of a freely acting subtorus in $T^s$ is called the \emph{Buchstaber number} of $K$; it satisfies the double inequality: $1\leq s(K)\leq s-n$. If $\Sigma$ is a normal fan of a convex simple $n$-dimensional polytope $P$ with $s$ facets and $s(\partial P^*)=s-n$, then $X_{\Sigma,\Lambda}$ is a \emph{quasitoric manifold} of $P$. This justifies the term 'generalized quasitoric manifold' we suggested above. Naturally, a quasitoric manifold serves as a key example and the most important particular case of a generalized quasitoric manifold.  
\end{remark}

A closed connected codimension-$k$ submanifold of $X_{\Sigma,\Lambda}$ is called \emph{characteristic} if it is fixed pointwise by a subtorus of dimension $k$ in $T^n$. Any such can be uniquely represented as a transverse intersection of $k$ characteristic submanifolds of codimension-2. It is easy to see that the face poset of characteristic submanifolds in $X_{\Sigma,\Lambda}$ is isomorphic to $K_\Sigma$.

In what follows we will always assume that our generalized quasitoric manifolds are \emph{omnioriented}; as in the case of a quasitoric manifold, we say that $X_{\Sigma,\Lambda}$ is omnioriented if an orientation is specified for $X_{\Sigma,\Lambda}$ and for each of the $s$ codimension-2 characteristic submanifolds $D_i$. The choice of this extra data is convenient for two reasons. First, it allows us to view the circle fixing $D_i$ as an element in the lattice $N=\Hom(S^1,T^n)\simeq\Z^n$. But even more importantly, the choice of omniorientation defines the fundamental class $[X_{\Sigma,\Lambda}]$ of $X_{\Sigma,\Lambda}$ as well as cohomology classes $[D_i]$ dual to the characteristic submanifolds. 

We further assume that $\Sigma \subset \R^n$ and $N_\R$ are endowed with orientation. This defines a sign for each collection of rays $\rho_{i_1},\ldots,\rho_{i_n}$ forming a maximal cone of $\Sigma$ in the following way. Let $I=\{i_1,\ldots,i_n\}$ be a set of indices ordered such that the collection of rays $\rho_{i_1},\ldots,\rho_{i_n}$ is positively oriented in $\R^n$, then
\[
\s(I)= \det(\Lambda(\rho_{i_1}),\ldots, \Lambda(\rho_{i_n})) = \pm 1.
\]
It is easy to see, that for codimension 2 characteristic submanifolds $D_{i_1},\ldots,D_{i_n}$ in a generalized quasitoric manifold one has
\[
[D_{i_1}]\cdots[D_{i_n}] = 
\begin{cases}
\s(I) \quad\text{ if } \rho_{i_1}\ldots,\rho_{i_n} \text{ form a cone in } \Sigma\\
0 \qquad \qquad\text{otherwise} 
\end{cases}
\]
\subsection{Cohomology ring: Stanley-Reisner-type description}

In this subsection we give a description of cohomology rings of generalized quasitoric manifolds based on the notion of a Stanley-Reisner ring of a simplicial complex.

Let $K$ be a simplicial complex on the vertex set $[s]$. The quotient algebra
$$
\Z[K]:=\Z[x_1,\ldots,x_s]/I,
$$
where $I$ is the monomial ideal generated by all square-free monomials $x_{i_1}\ldots x_{i_t}$ with $\{i_1,\ldots,i_t\}\notin K$, is called the \emph{Stanley-Reisner ring} (or, \emph{face ring}) of $K$ (with integer coefficients). In what follows we assume that $\Z[K]$ is a graded algebra with $\deg x_i=2, 1\leq i\leq s$. Denote by $\Z[\Sigma]:=\Z[K_\Sigma]$.

Recall that any generalized quasitoric manifold $X_{\Sigma,\Lambda}$ with $\dim\Sigma=n$ is a $2n$-dimensional \emph{locally standard torus manifold}; that is, a smooth connected closed orientable $2n$-dimensional manifold with an effective smooth locally standard action of the compact torus $T^n$ such that the fixed point set $M^T$ is non-empty and finite. 

Moreover, since the dual $Q^{\vee}$ to the manifold with corners $Q:=X_{\Sigma,\Lambda}/T^n$ is isomorphic to the starshaped sphere $K_\Sigma$, the orbit space $Q$ is a \emph{homology polytope}; that is, a manifold with corners in which all faces, including $Q$ itself are acyclic, and all nonempty intersections of faces are connected. 

The next result is a particular case of a theorem due to Masuda and Panov~\cite{panovmasuda}, who computed (equivariant) cohomology rings of torus manifolds. 
It describes the (equivariant) cohomology ring of $X_{\Sigma,\Lambda}$ in terms of the Stanley-Reisner ring $\Z[K_\Sigma]$ of $K_\Sigma$.

\begin{theorem}\label{cohgenqtrmfd}
\begin{enumerate}
\item Equivariant cohomology ring $H^*_{T^n}(X_{\Sigma,\Lambda})$ is isomorphic to $\Z[K_\Sigma]$;
\item Ordinary cohomology ring $H^*(X_{\Sigma,\Lambda})$ is isomorphic to $\Z[K_\Sigma]/J$, where 
$$
J = \rleft\langle \sum_{i=1}^s \chi(v_i)x_i \colon \chi \in M \rright\rangle\text{ and }v_i:=\Lambda(\rho_i), 1\leq i\leq s.
$$ 
\end{enumerate}
\end{theorem}

When $\Sigma$ is a normal fan of a convex simple polytope $P$ with $s(\partial P^*)=s-n$, these results were obtained by Davis and Januszkiewicz~\cite{davis1991convex}. This description of the cohomology ring yields an algorithm to compute products of top-degree cohomology classes in $H^*(X_{\Sigma,\Lambda},\R)$. It can be done as follows.

As $R_{\Sigma,\Lambda}:=H^*(X_{\Sigma,\Lambda},\R)$ is generated in degree 1, it suffices to consider monomials $x_{i_1}^{k_1}\ldots x_{i_t}^{k_t}$.
Recall that the graded piece of top degree of $R_{\Sigma,\Lambda}$ is one-dimensional and is generated by $x_{i_1}\cdots x_{i_n}$, for any collection $i_1,\ldots, i_n$ of indices such that the rays $\rho_{i_1}, \ldots, \rho_{i_n}$ generate a full-dimensional cone in $\Sigma$.
Indeed, all such monomials $x_{i_1}\cdots x_{i_n}$ yield the same element in $R_{\Sigma,\Lambda}$.
Therefore, the evaluation of a monomial $x_{i_1}^{k_1}\ldots x_{i_t}^{k_t}$ in the top degree amounts to expressing it as a linear combination of square free monomials.

For a monomial $x_{i_1}^{k_1}\cdots x_{i_t}^{k_t}$, let $m=\sum_{i=1}^t(k_i-1)$ be its \emph{multiplicity}, so that being square free is equivalent to having multiplicity $0$.
To simplify notation, consider the monomial $x_1^{k_1}\cdots x_t^{k_t}$ (always possible by reordering the variables).
Suppose $x_1^{k_1}\cdots x_t^{k_t}$ is a monomial with multiplicity $m>0$ and $k_1>1$.
The goal is to express it as a linear combination of monomials of smaller multiplicity.
If $\rho_1,\ldots, \rho_t$ do not form a cone, then this monomial is equal to zero in $R_{\Sigma,\Lambda}$ and we are done.
Otherwise, the set $\{v_1,\ldots, v_t\}$ can be extended to a lattice basis of $N$, and therefore there is $\chi \in M$ such that $\chi(e_1)=1$ and  $\chi(e_j)=0$, for $j=2,\ldots,t$.
Note that $\chi$ induces a linear relation in $J$, so that we obtain
\[
  x_1^{k_1}\ldots x_t^{k_t} = x_1^{k_1-1} \ldots x_t^{k_t} \cdot \rleft(-\sum_{k=t+1}^s \chi(v_k)x_k\rright) \text{.}
\]
This is a linear combination of monomials of multiplicity $m-1$.
Applying this procedure inductively, we end up with a linear combination of square free monomials, and thus obtain an evaluation of $x_1^{k_1} \cdots x_t^{k_t}$.

\subsection{Cohomology ring: volume polynomial description}

Another description of the cohomology ring of $X_{\Sigma,\Lambda}$ comes from the volume polynomial on the space of analogous  multi-polytopes. To a simplicial fan $\Sigma$ with the characteristic map $\Lambda$ one can associate a vector space of analogous  multi-polytopes $\Pm_{\Sigma,\Lambda}$. Multi-fans and multi-polytopes were introduced in \cite{multifan}. 

Similar to a space of virtual polytopes, the space $\Pm_{\Sigma,\Lambda}$ can be identified with the space of $\Lambda$-piecewise linear functions on $\Sigma$. 

\begin{definition}
A function $h:\Sigma\to \R$ is called \emph{$\Lambda$-piecewise linear} if its restriction $f|_\sigma$ to any cone $\sigma$ of $\Sigma$ is a pullback of a linear function on $N$:
\[
f|_\sigma = (\Lambda^* \chi)|_\sigma, \text{ for some } \chi\in M_\R.
\]
A $\Lambda$-piecewise linear function is called \emph{integer}, if for any cone $\sigma$, the corresponding linear function $\chi_\sigma \in M$ is integer.
\end{definition}

We define the \emph{space of analogous  multi-polytopes} as the space of all $\Lambda$-piecewise linear functions on $\Sigma$:
\[
  \Pm_{\Sigma,\Lambda}= \rleft\{ h \colon \Sigma \to \R \colon \Lambda- \text{piecewise linear function on} \; \Sigma \rright\}. 
\]
Any element of $\Pm_{\Sigma,\Lambda}$ is determined by a choice of a point on each ray from $\Sigma(1)$. 

The above description generalizes the description of virtual polytopes via their support functions. In Section~4 we give a more detailed introduction to multi-polytopes. In particular, we discuss their geometric description as convex chains. By analogy with virtual polytopes, we will usually denote a multi-polytope by $\Delta$ and the corresponding $\Lambda$-piecewise linear function by $h_\Delta$.

Let $X_{\Sigma, \Lambda}$ be a generalized quasitoric manifold and let $D_1,\ldots, D_s$ be characteristic submanifolds of $X_{\Sigma, \Lambda}$ of codimension~$2$. Since $X_{\Sigma,\Lambda}$ is omnioriented, every multi-polytope $\Delta \in \Pm_{\Sigma,\Lambda}$ defines a degree 2 cohomology class of $X_{\Sigma, \Lambda}$ via
\[
\Delta \mapsto  \sum_{i=1}^s h_\Delta(e_i) [D_i],
\]
where $[D_i]\in H^2(X_{\Sigma, \Lambda},\R)$ is a class Poincar\'e dual to the characteristic submanifold $D_i$. For quasitoric manifolds, it was shown in \cite{davis1991convex} that the cohomology ring $H^*(X_{\Sigma, \Lambda},\R)$ is generated by the classes $[D_1],\ldots,[D_s]$ and, in particular, cohomology is generated by its degree-two part.

Commutative graded rings with Poincar\'e duality generated in degree 1 admit a convenient description via Macaulay's inverse systems. The following theorem is implicit in \cite{KP} and forms a key idea to get a description of the cohomology ring $H^*(X_\Sigma, \R)$.
The following version is taken from \cite[Theorem 1.1]{KavehVolume}.

\begin{theorem}
  \label{PKh}
  Suppose $A = \bigoplus_{i=0}^n A_i$ is a graded finite dimensional commutative algebra over a field $\K$ of characteristic $0$ such that $A$ is generated (as an algebra) by the elements $A_1$ of degree one, $A_0 \simeq A_n \simeq \K$, and the bilinear map $A_i \times A_{n-i} \to A_n$ is non-degenerate for any $i = 0, \ldots, n$ (Poincar\'{e} duality).
  Then 
  \[
    A\simeq \K[t_1, \ldots, t_r]/\{p(t_1, \ldots, t_r) \in \K[t_1, \ldots, t_r] \colon p(\tfrac{\partial}{\partial x_1}, \ldots, \tfrac{\partial}{\partial x_r}) f(x_1, \ldots, x_r) = 0 \}
  \]
  where we identify $A_1$ with $\K^r$ via a basis $v_1, \ldots, v_r$ and define $f \colon A_1 \simeq \K^r \to \K$ as the polynomial given by $f(x_1,\ldots, x_r) = (x_1v_1 + \ldots + x_rv_r)^n \in A_n \simeq k$.
\end{theorem}

\begin{remark}
  \label{rem:diff-ops}
  Instead of identifying $A_1$ with $\K^r$, the previous theorem accepts a basis-free formulation in terms of $\Diff(A_1)$, the ring of differential operators with constant coefficients:
  $A \simeq \Diff(A_1) / \{ p \in \Diff(A_1) \colon p (f) = 0\}$.
  %In Section~\ref{commalg}, we summarize recent results on algebras with Poincar\'e duality which are not necessarily generated in degree $1$.
  In particular, we formulate Theorems~\ref{thm-nsdquatient} and~\ref{potential} which generalize  Theorem~\ref{PKh}.
\end{remark}

Theorem~\ref{PKh} together with Theorem~\ref{cohgenqtrmfd} imply that to compute the cohomology ring of $X_{\Sigma,\Lambda}$ it is enough to know the self-intersection index of characteristic submanifolds $D_1,\ldots, D_s$. 
This was done by Ayzenberg and Masuda~\cite{ayzenberg2016volume}. They defined the notion of a volume polynomial of a multi-polytope and obtained the following generalization of the \BKK theorem to (generalized) quasitoric manifolds.

\begin{theorem}\label{thm:MABKK}
  Let $X_{\Sigma,\Lambda}$ be a (generalized) quasitoric manifold and let $\Delta \in \Pm_{\Sigma,\Lambda}$ be a multi-polytope. Then 
  \[
  \left(\sum_{i=1}^s h_\Delta(e_i) [D_i]\right)^n = n!\cdot \Vol(\Delta).
  \]
\end{theorem}

\begin{corollary}
   Let $X_{\Sigma,\Lambda}$ be a generalized quasitoric manifold, then the cohomology ring $H^*(X_{\Sigma,\Lambda},\R)$ can be computed as 
   \[
   H^*(X_{\Sigma,\Lambda},\R) \simeq \Diff(\Pm_{\Sigma,\Lambda})/\Ann(\Vol),
   \]
   where $\Ann (\Vol)$ is the ideal of differential operators annihilating the volume polynomial on $\Pm_{\Sigma,\Lambda}$.
\end{corollary}

\subsection{Cohomology ring: Brion-type description}
In this subsection we give the third description of the cohomology ring of a quasitoric manifold via the ring of piecewise polynomial functions on the multi-fan $(\Sigma, \Lambda)$.

For a multi-fan $(\Sigma,\Lambda)$, we call a continuous function $f:\Sigma \to \R$ ($\Lambda$-)\emph{piecewise polynomial} if its restriction $f|_\sigma$ to any cone $\sigma$ of $\Sigma$ is a pullback of a polynomial on $N_\R$ via $\Lambda$:
\[
f|_\sigma = \Lambda^* g_\sigma, \quad g_\sigma \in \R[M].
\]
The set of all piecewise polynomial functions on $(\Sigma,\Lambda)$ forms an $\R$-algebra with respect to pointwise addition and multiplication. Let us denote the algebra of piecewise polynomial functions on $(\Sigma,\Lambda)$ by $\PP_{\Sigma, \Lambda}$.

\begin{theorem}\label{thm:Brion}
  As before, let $X_{\Sigma,\Lambda}$ be a generalized quasitoric manifold given by the multi-fan $(\Sigma,\Lambda)$. 
  Then equivariant cohomology ring of $X_{\Sigma,\Lambda}$ is given by the ring of piecewise polynomial functions on $(\Sigma,\Lambda)$:
  \[
  H^*_T(X_{\Sigma,\Lambda},\R) \simeq \PP_{\Sigma,\Lambda}.
  \]
  Furthermore, the cohomology ring $H^*(X_{\Sigma,\Lambda},\R)$ is given by
  \[
  H^*(X_{\Sigma,\Lambda},\R) \simeq \PP_{\Sigma,\Lambda}/\langle M\rangle.
  \]
\end{theorem}

\begin{proof}
  A generalized quasitoric manifold $X_{\Sigma,\Lambda}$ has only finitely many $T$-fixed points, and for each $p\in X_{\Sigma,\Lambda}^T$ the tangent space $T_pX_{\Sigma,\Lambda}$ is decomposed as a direct sum of irreducible representations of $T$ given by linearly independent characters. Moreover, $H^{odd}(X_{\Sigma,\Lambda})=0$ hence $X_{\Sigma,\Lambda}$ is an equivariantly-formal GKM manifold. The description of equivariant cohomology of $X_{\Sigma,\Lambda}$ follows from computation of equivariant cohomology of the corresponding GKM-graph (\cite{goresky1998equivariant, kuroki2009introduction}). 
  
  Indeed, the fixed points of $T$-action on $X_{\Sigma,\Lambda}$ are in bijection with maximal cones of $\Sigma$ and for the two maximal cones $\sigma_1, \sigma_2$ connected by a codimension 1 face, the label of the corresponding edge in the GKM-graph is given by a lattice generator of $(\sigma_1\cup \sigma_2)^\perp$. Hence we have:
  \[
  H^*_T(X_{\Sigma,\Lambda}^T,\R) \simeq  \left\{(f_p)_{p\in X_{\Sigma,\Lambda}}\in \Sym(M)\times\ldots\times\Sym(M) \,\Big|\, f_{p_{\sigma_1}}|_{\sigma_1\cap\sigma_2}=f_{p_{\sigma_2}}|_{\sigma_1\cap\sigma_2}\right\} \simeq \PP_{\Sigma,\Lambda}
  \]
  and
  \[
    H^*(X_{\Sigma,\Lambda},\R) \simeq \PP_{\Sigma,\Lambda}/\langle M\rangle.
  \]
\end{proof}

%%%%%%%%%%%%%%%%%%%%%%%%%%%%%%%%%%%%%%%%%%%%%%%%%%%%%%%%%%%%%%%%%%%%%%%%%%%%%%%%%%%%%%%%%%%%%%%%%%%%%%%%%%%%%%%%%%%%%%%%%%%%%%%%%%%%%%%%%%%%%%%%%%%%%%%%%

\section{BKK theorems for quasitoric bundles}
\label{sec:general}

In this section, we first collect several facts on quasitoric bundles which will be used below. Then we formulate the main results of this paper: Theorems~\ref{BKK} and~\ref{BKK1}.
Details and proofs will be given in Section~\ref{BKKproof}, and in Section~\ref{sec:applications} these theorems will be used to describe cohomology rings of toric bundles.

%%%%%%%%%%%%%%%%%%%%%%%%%%%%%%%%%%%%%%%%%%%%%%%%%%%%%%%%%%%%%%%% 

\subsection{Preliminaries on quasitoric bundles}
\label{sec:toric_bundles}

Let $G$ be a topological group and let $p \colon E \to B$ be a \emph{principal $G$-bundle} over a topological space $B$.
Recall that to any principal $G$-bundle $p \colon E \to B$ and any topological space $X$ equipped with a continuous action by $G$, one associates a fiber bundle by introducing a (right) action on the product $E \times X$:
\[
  (e, x) \cdot g \coloneqq (e\cdot g, g^{-1}\cdot x) \text{.}
\]
The \emph{associated fiber bundle} is given as the quotient $E \times_G X \coloneqq (E \times X) / G$.
It is a fiber bundle with fiber $X$.
If $G = T\simeq(S^1)^n$ is an $n$-dimensional torus, then a $T$--principal bundle is also called a \emph{torus bundle}.

Crucial to the understanding of the cohomology of fiber bundles is the following theorem (see, for instance,  \cite[Theorem 4D.1]{Hatcher}):
\begin{theorem}[Leray--Hirsch]
  \label{thm:Leray-Hirsch}
  Let $E$ be a fiber bundle with fiber $F$ over a compact manifold $B$.
  If there are global cohomology classes $u_1,\ldots , u_r$ on $E$ whose restrictions $i^*(u_i)$ form a basis for the cohomology of each fiber $F$ (where $i \colon F \to E$ is the inclusion), then we have an isomorphism of vector spaces:
  \[
    H^*(B, \R)\otimes H^*(F, \R) \to H^*(E, \R); \quad \sum_{i,j} b_i \otimes i^*(u_j) \mapsto \sum_{i,j} p^*(b_i) \cdot u_j\text{.}
  \]
\end{theorem}

\begin{corollary}
  \label{cor:cohmodule}
  If $T$ is a torus, $p \colon E \to B$ a torus bundle as in Theorem~\ref{thm:Leray-Hirsch}, and $X$ a generalized $T$--quasitoric manifold, then as a group the cohomology of $E_X = E \times_T X$ is given by
  \[
    H^*(E_X, \R) \simeq H^*(B, \R) \otimes H^*(X, \R) \text{.}
  \]
\end{corollary}
\begin{proof}
  As a group, the cohomology of generalized quasitoric manifolds is generated by classes Poincar\'{e} dual to characteristic submanifolds.
  For any characteristic submanifold $X_\sigma$ of $X$, let $E_{\sigma} = E \times_T X_\sigma$ be the associated bundle which is a submanifold of $E_X$. Let us choose a linear basis $X_{\sigma_1}, \ldots, X_{\sigma_r}$ of $H^*(X,\R)$, then the cohomology classes $u_1, \ldots, u_r$ Poincar\'{e} dual to $E_{\sigma_1},\ldots, E_{\sigma_r}$ satisfy the condition of Theorem~\ref{thm:Leray-Hirsch}.
  The statement follows.
\end{proof}

Corollary~\ref{cor:cohmodule} yields a description of the cohomology \emph{group} of $E_X$.
Crucial for this description is the map which associates to any characteristic submanifold $X_\sigma$ of $X$ the Poincar\'{e} dual of the corresponding $T$-invariant submanifold of $E_X$.
By restricting this map to codimension $2$ characteristic manifolds and extending using linearity, we obtain
\[
  \rho\colon \Pm_{\Sigma,\Lambda} \to H^2(E_X, \R) \qquad \text{(where $(\Sigma,\Lambda)$ is the multi-fan of $X = X_{\Sigma,\Lambda}$)}
\]
which plays an important role in our description of the cohomology ring of $E_X$.
We provide more details for $\rho$.
Recall that $\rho_1,\ldots,\rho_r$ denote the rays of $\Sigma$ with  
\[
\Lambda(\rho_1)=e_1,\ldots, \Lambda(\rho_r)=e_r.
\]
Let $D_1,\ldots,D_r$ be the corresponding codimension 2 characteristic submanifolds in $X$.
We also write $D_i$ for the submanifold $E \times_T D_i$ of $E_X$, note that it comes with orientation induced by omniorientation of $X$.
Then for $\Delta \in \Pm_{\Sigma,\Lambda}$, we have:
\[
  \rho(\Delta) = \sum_{i=1}^r h_\Delta(e_i) [D_i] \in  H^2(E_X, \R),
\]
where $[D_i]$ is the class Poincar\'{e} dual to $D_i\subseteq E_X$ and $h_\Delta \colon N_\R \to \R$ is the support function of multi-polytope~$\Delta$.

As before, any character $\mu\in M$ of $T$ can be viewed as a multi-polytope, since $\Lambda^* \mu$ is a $\Lambda$-piecewise linear function on $\Sigma$. Therefore, the map $\rho$ defines a group homomorphism $\rho: M\to H^2(E_X,\Z)$, which we denote by the same symbol. We extend this map by linearity to a map of vector spaces $\rho: M_\R \to H^2(E_X,\R)$

On the other hand, any character $\lambda \in M$ defines a one--dimensional representation $\C_\lambda$ of $T$, namely $t \cdot z = \lambda(t) z$ for $t \in T$, and $z \in \C_\lambda$.
If $\cL_\lambda$ denotes the associated complex line bundle on $B$, i.e. $\cL_\lambda\simeq E\times_T\C_\lambda$, then $\cL_{\lambda+\mu}=\cL_\lambda\otimes\cL_\mu$, and thus we obtain a group homomorphism:
\[
  c\colon M\to H^2(B,\Z), \quad \lambda \mapsto c_1(\cL_\lambda),
\]
where $c_1(\cL_\lambda)$ is the first Chern class.
By linearity, we extend the homomorphism to a map of vector spaces:
\[
  c\colon  M_\R \to H^2(B,\R).
\]

The following observation about connection between $\rho(\cdot)$ and $c(\cdot)$ will be crucial for our approach.

\begin{proposition}\label{cherneq}
  Let $X$ be a generalized quasitoric bundle given by a characteristic pair $(\Sigma, \Lambda)$ and let $p\colon E_X \to B$ be as before.
  Then for any character $\lambda \in M$:
  \[
    p^* c(\lambda)= \rho(\lambda),
  \]
  where on the right hand side of the equality $\lambda$ is regarded as a virtual polytope.
\end{proposition}
\begin{proof}
  A character $\lambda$ of $T$ defines an equivarient line bundle $L_\lambda$ on the generalized quasitoric manifold $X$. The associated fiber bundle $E\times_TL_\lambda$ is a line bundle over $E_X$. Moreover, the Chern class of $E\times_TL_\lambda$ given by $c_1(E\times_TL_\lambda)=\rho(\lambda)$. It is easy to see that $E\times_TL_\lambda$ is a pullback of line bundle $\cL_\lambda$ on $B$
  \[
  E\times_TL_\lambda = p^*\cL_\lambda.
  \]
  Hence the proposition follows from naturality of Chern classes.
\end{proof}

We will finish this subsection with the following observation. The map of lattices $c\colon M\to H^2(B,\Z)$ is an invariant that uniquely describes torus bundles in the topological category.
Indeed, principle $G$-bundles over a smooth manifold $B$ are classified by the homotopy classes of maps $f:B\to BG$ to the classifying space. Since  the classifying space of a compact torus $T$ is homotopy equivalent to $\C\p^\infty\times \ldots \times \C\p^\infty$ and, in particular, is an Eilenberg–MacLane space $K(\Z^n,2)$,  homotopy classes of maps are uniquely determined by 
\[
c:=f^*: H^2(BT)\simeq M(T) \to H^2(B)
\]
We arrive at the following result.

\begin{proposition}
  \label{prop:toricbun}
  Let $B$ be a closed oriented manifold and $T$ be a compact torus with  lattice of algebraic characters $M(T)$.
  Then $T$-principal  bundles over $B$ are in bijection with homomorphisms $c\colon M(T)\to H^2(B,\Z)$.
\end{proposition}

%%%%%%%%%%%%%%%%%%%%%%%%%%%%%%%%%%%%%%%%%%%%%%%%%%%%%%%%%%%%%%%%%

\subsection{\BKK theorems for quasitoric bundles}
\label{sec:bkk-tb}
As before, let $p\colon E\to B$ be a principal torus bundle with respect to a torus $T\simeq (S^1)^n$ over a compact smooth orientable manifold $B$ of real dimension $k$.
Let $M$ be the character lattice of $T$ and $(\Sigma, \Lambda)$ be a complete multi-fan which gives rise to a generalized quasitoric manifold $X = X_{\Sigma,\Lambda}$.
Let $E_X$ be the total space of the associated quasitoric bundle.
Note that $E_X$ is a compact smooth orientable manifold of real dimension $k+2n$.
To keep notation simple we denote the projection map of the toric bundle by $p\colon E_X \to B$ as well.

Our main theorems show that a choice of a natural number $i\leq\tfrac{k}{2}$ and $\gamma\in H^{k-2i}(B, \R)$ gives rise to a \mbox{\BKK-type} theorem.
First, we define two functions $I_\gamma$ and $F_\gamma$ on $\Pm_{\Sigma,\Lambda}$ as follows.

Let $f_\gamma\colon M_\R \to \R$ be given by 
\[
  f_\gamma(x)= \langle c(x)^{i} \cdot \gamma,[B]\rangle\text{,}
\]
where ``$\cdot$'' denotes the cup product of the cohomology ring $H^*(B,\R)$. Since $c\colon M_\R\to H^2(B,\R)$ is a linear map, $f_\gamma$ is a homogeneous polynomial of degree $i$ on $M_\R$.
This leads to the definition of $I_\gamma$:
\[
  I_\gamma\colon \Pm_{\Sigma,\Lambda}\to \R; \quad I_\gamma(\Delta) \coloneqq \int_\Delta f_\gamma(x) \diff\mu  \quad \text{for} \; \Delta\in \Pm_{\Sigma,\Lambda},
\]
where $\mu$ denotes the Lebesgue measure on $M_\R$ normalized with respect to the lattice $M$, i.e. a cube spanned by an affine lattice basis of $M$ has volume $1$. By Theorem~\ref{ultimatepoly}, $I_\gamma$ is well defined homogeneous polynomial of degree $n+i$ on $\Pm_{\Sigma,\Lambda}$.

Recall the definition of $\rho\colon \Pm_{\Sigma,\Lambda}\to H^{2}(E_X, \R)$ from Section~\ref{sec:toric_bundles}.
This leads to the definition of the function~$F_\gamma$:
\[
  F_\gamma\colon \Pm_{\Sigma,\Lambda} \to \R; \quad F_\gamma(\Delta) \coloneqq \langle \rho(\Delta)^{n+i}\cdot p^*(\gamma), [E_X]\rangle.
\]
Clearly, $F_\gamma$ is a homogeneous polynomial of degree $n+i$ on $\Pm_{\Sigma,\Lambda}$.

The main result of this section is the following analog of the \BKK theorem for quasitoric bundles.
Indeed, it expresses certain intersection numbers of cohomology classes as mixed integrals.

\begin{theorem}
  \label{BKK}
  The polynomials $I_\gamma$ and $F_\gamma$ are proportional with coefficient of proportionality given by:
  \[
    (n+i)!\cdot I_\gamma(\Delta)=i!\cdot F_\gamma(\Delta) \qquad \text{for any $\Delta\in\Pm_{\Sigma,\Lambda}$.}
  \]
  In particular, the polarizations of $I_\gamma$ and $F_\gamma$ are proportional multilinear forms, i.e. for any $\Delta_1,\ldots,\Delta_{n+i}\in \Pm_{\Sigma,\Lambda}$
  \[
    (n+i)!\cdot I_\gamma(\Delta_1,\ldots,\Delta_{n+i})=i!\cdot F_\gamma(\Delta_1,\ldots,\Delta_{n+i}).
  \]
\end{theorem}

For the reader's convenience, we recall the concept of polarization (or equivalently mixed integrals). Let $V$ be a  vector space. Recall that for a homogeneous polynomial $f\colon V \to \R$ of degree $m$, the \emph{polarization of $f$} is the unique symmetric multilinear form $g\colon V^m \to \R$ such that $g(v,\ldots,v)=f(v)$. 
It is a well-known fact that for any vector space $V$ and any homogeneous polynomial $f$ of degree $m$, the polarization exists and can be defined as follows:
\begin{equation}\label{eq:polarisation_formula}
  g(v_1,\ldots,v_m) = \frac{1}{m!}L_{v_1}\ldots L_{v_m}f,
\end{equation}
where by $L_v$ we denote the Lie derivative in direction of $v$. For a homogeneous polynomial $f\colon M_\R\to \R$, we will call the polarization of $I_f$ on the space of analogous  multi-polytopes $\Pm_{\Sigma,\Lambda}$  the \emph{mixed integral of $f$}.

We conclude this section with an alternative interpretation of Theorem~\ref{BKK} which can be favourable for certain applications.
By the Leray-Hirsch Theorem (see Corollary~\ref{cor:cohmodule}), $H^*(E_X, \R) \simeq H^*(B, \R)\otimes H^*(X, \R)$, and so for $\Delta \in \Pm_{\Sigma,\Lambda}$, the cycle $\rho(\Delta)^{n+i}\in H^{2n+2i}(E_X, \R)$ can be written as
\[
  \rho(\Delta)^{n+i} =  b_{2n+2i}\otimes x_{0} +  b_{2n+2i-2}\otimes x_{2} +\ldots + b_{2i+2}\otimes x_{2n-2} +  b_{2i}\otimes x_{2n},
\]
with $b_{s}\in H^s(B, \R)$ and $x_{r}\in H^{r}(X, \R)$.
As $X$ is a quasitoric manifold, its cohomology groups in odd degrees vanish, and therefore $x_{2k+1}=0$ for any $k$.
If $x_{2n}$ is normalized such that it is dual to a point, we call $b_{2i}$ the \emph{horizontal part} of $\rho(\Delta)^{n+i}$.
Equivalently, the horizontal part $b_{2i}$ of $\rho(\Delta)^{n+i}$ is the unique class in $H^{2i}(B,\R)$ such that
\[
  \langle\rho(\Delta)^{n+i}\cdot p^*(\eta),[E_X]\rangle = \langle b_{2i} \cdot \eta, [B]\rangle  \text{,}
\]
for any $\eta\in H^{k-2i}(B,\R)$.
Then Theorem~\ref{BKK} accepts the following reformulation.

\begin{theorem}
  \label{BKK1}
  For any $\Delta \in \Pm_{\Sigma,\Lambda}$, the horizontal part of $\rho(\Delta)^{n+i}$ can be computed as
  \[
    b_{2i} = \frac{(n+i)!}{i!} \int_\Delta c(x)^i \diff x \text{.}
  \]
\end{theorem}

Note that $c(\cdot)^i \colon M_\R  \to H^{2i}(B,\R)$ is a vector valued map whose components (after choosing suitable coordinates) are given by homogeneous polynomials of degree $i$.
Thus the integral in Theorem~\ref{BKK1} exists.
Furthermore, although we show that Theorem~\ref{BKK} implies Theorem~\ref{BKK1}, in fact they are equivalent.

\begin{proof}
  Since $H^*(B, \R)$ satisfies Poincar\'{e} duality, it suffices to check that for any $\gamma \in H^{k-2i}(B, \R)$, we have
  \[
   \left\langle i! \cdot \gamma \cdot b_{2i},[B]\right \rangle= (n+i)! \left\langle\gamma \cdot \int_\Delta c(x)^i \diff x, [B]\right \rangle,
  \]
  Recall from Theorem~\ref{thm:Leray-Hirsch}, that there are $u_1, \ldots, u_r \in H^*(E_X, \R)$ such that the restrictions $i^*(u_i)$ form a basis for the cohomology of each fiber $X$ where $i \colon X \to E_X$ is the inclusion.
  In particular, there are $y_i \in \R u_1 \oplus \ldots \oplus \R u_r$ such that $i^*(y_i) = x_i$.
  Since $x_{2n}$ is Poincar\'{e} dual to a point (say to a torus fixed point $x \in X$), it follows that $y_{2n} = E \times_T \{x\}$, i.e. $y_{2n}=[S]^*$ is the class Poincar\'e dual to a section of $p$.
  
  Let $[\pt]^* \in H^k(B, \R)$ be the class dual to a point in $B$.
  For $\gamma \in H^{k-2i}(B, \R)$, let $\gamma \cdot b_{2i} = a \cdot [\pt]^*$ for some real number $a \in \R$.
  Then,
  \begin{align*}
    p^*(\gamma) \cdot \rho(\Delta)^{n+i} &= p^*(\gamma) \cdot (p^*(b_{2n+2i}) \cdot y_0 + p^*(b_{2n+2i-2}) \cdot y_2+ \ldots + p^*(b_{2i+2}) \cdot y_{2n-2}+p^*(b_{2i}) \cdot y_{2n})\\
    &= p^*(\gamma \cdot b_{2i}) \cdot y_{2n} = a \cdot p^*([\pt]^*) \cdot E\times_T \{x\} =a \cdot [X]^* \cdot [S]^*= a \in H^{2n+k}(E_X, \R) \simeq \R \text{,}
  \end{align*}
  where $[X]^*$ is the class dual to a fiber of $p$.
  Hence, we get
  \[
    \langle\gamma \cdot b_{2i},[B]\rangle = \langle p^*(\gamma) \cdot \rho(\Delta)^{n+i},[E_X]\rangle  = F_\gamma(\Delta) \text{.}
  \]
  On the other hand,
  \[
  \left \langle \gamma \cdot \int_\Delta c(x)^i \diff x,[B]\right\rangle = \int_\Delta \langle\gamma \cdot c(x)^i ,[B]\rangle\diff x = I_\gamma(\Delta).
  \]
  The statement follows by Theorem~\ref{BKK}.
\end{proof}

Like Theorem~\ref{BKK}, Theorem~\ref{BKK1} admits a polarized version.

%%%%%%%%%%%%%%%%%%%%%%%%%%%%%%%%%%%%%%%%%%%%%%%%%%%%%%%%%%%%%%%%%

\section{Proof of the \BKK theorems}
\label{BKKproof}

This section is devoted to the proof of Theorem~\ref{BKK}. Before we begin with the proof, let us summarize needed results on the polynomial measures on the space of analogous  multi-polytopes. For the details and proofs please see \cite{quasimulti,key}.

As before, let $(\Sigma, \Lambda)$ be a characteristic pair and let $\Sigma(1)=\{\rho_1,\ldots,\rho_s\}$ be the set of rays of $\Sigma$ as before. To a pair $(\Sigma, \Lambda)$ we associate the space of \emph{multi-polytopes} $\Pm_{\Sigma,\Lambda}$ which is identified with the space of $\Lambda$-piecewise linear functions on $\Sigma$.

Note that since the fan $\Sigma$ is simplicial, a $\Lambda$-piecewise linear function on $\Sigma$ is uniquely determined by its values on the unit vectors in the directions of rays of $\Sigma$. Therefore the space  $\Pm_{\Sigma,\Lambda}$ is naturally isomorphic to $\R^s$ with the natural coordinates $h_1,\ldots,h_s$ given by evaluation of  $\Lambda$-piecewise linear functions on the unit vectors in the directions of rays $\rho_1,\ldots,\rho_s$.

There is a correspondence between the cones of $\Sigma$ and faces of the multi-polytope $\Delta \in \Pm_{\Sigma,\Lambda}$. Indeed, for each cone $\sigma\in \Sigma$ of dimension $i$ there exists an affine subspace $M_\sigma\subset M_\R$ of codimension $i$ with a multi-polytope $\Delta_\sigma$ in it. We will only use this correspondence in the case of maximal dimensional cones. Let $\sigma$ be a maximal cone of $\Sigma$. Without loss of generality we may assume that $\sigma$ is generated by the rays $\rho_1,\ldots,\rho_n$. Then the corresponding affine subspace $M_\sigma$ is zero-dimensional given by the system of linear equations
\[
\langle \Lambda(\rho_i), x \rangle = h_{\Delta}(e_i), \quad i = 1,\ldots, n,
\]
and we identify the corresponding vertex of $\Delta$ with the unique solution of the system above.

The main result we will need is the following construction. Let $f:M_\R \to \R$ be a homogeneous polynomial of degree $d$, then there is a well-defined integration functional on the space of virtual polytopes:
\[
I_f(\Delta) = \int_{\Delta} f \diff \mu, \quad \text{for } \Delta\in \Pm_{\Sigma,\Lambda}.\footnote{To give a proper construction one has to use a more geometric realization of multi-polytopes as convex chains. This is the approach we use in \cite{quasimulti}. However, here we will not give the actual definition of $I_f$: we will only formulate some of its properties as this is enough for the purposes of this paper.}
\]

The main result we will need is the following theorem.

\begin{theorem}
  \label{ultimatepoly}
  Let $f\colon M_\R \to\R$ be a homogeneous polynomial of degree $d$, then the function $I_f\colon \Pm_{\Sigma,\Lambda}\to \R$ is a homogeneous polynomial of degree $n+d$ on $\Pm_{\Sigma,\Lambda}$.
\end{theorem}

We will further need the following result on the functional $I_f$. Let  $\partial_i= \partial/\partial_{h_i}$ be the partial derivatives along the coordinate vectors of $\Pm_{\Sigma,\Lambda} \simeq \R^s$.

\begin{lemma}
  \label{Ider}
  Let $I = \{ i_1, \ldots, i_r\} \subseteq \{ 1, \ldots, s \}$ be a subset and $k_1, \ldots, k_r$ positive integers.
  Let $\Delta\in \Pm_{\Sigma,\Lambda}$ be a multi-polytope and $\rho_{i_1},\ldots,\rho_{i_r}$ do not span a cone in $\Sigma$, then we have
  \[
    \partial_{i_1}^{k_1} \cdots \partial_{i_r}^{k_r} \left(I_f|_{\Pm_\Sigma}\right)(\Delta) = 0 \text{.}
  \]
  However, if $r = n$ and $\rho_{i_1}, \ldots, \rho_{i_n}$ span a cone in $\Sigma$ dual to the vertex $A\in M_\R$, we have 
  \[
    \partial_I \left(I_f|_{\Pm_\Sigma}\right)(\Delta) = \s(I)  f(A) \cdot |\det (e_{i_1}, \ldots, e_{i_n})| \text{.}
  \]
\end{lemma}

\subsection{Proof of Theorem~\ref{BKK}}
For the reader's convenience, we recall the used notation.
Let $p \colon E \to B$ be a principal torus bundle with respect to an $n$-dimensional torus $T$ over a smooth compact orientable manifold $B$ of real dimension $k$.
The character lattice of $T$ we denote by $M$.
Let $\Sigma \subseteq \R^n$ be a simplicial complete fan with rays $\rho_1, \ldots, \rho_s$ and let $\Lambda:\Sigma(1) \to M$ be a characteristic map.
Let $X = X_{\Sigma,\Lambda}$ be the generalized quasitoric manifold corresponding to $(\Sigma,\Lambda)$ and denote the total space of the associated quasitoric bundle by $E_X$.
To keep notation simple we use the same notation $p \colon E_X \to B$ for the projection map of the toric bundle.
Fix $i \le \frac{k}{2}$ and a class $\gamma \in H^{k-2i}(B,\R)$.
The rays $\rho_i$ correspond to codimension $2$ characteristic submanifolds in $X$ which give rise to codimension $2$ submanifolds $D_i$ in $E_X$.
For a multi-polytope $\Delta \in \Pm_{\Sigma,\Lambda}$, we introduced
\[
  \rho(\Delta) = \sum_{i=1}^s h_\Delta(e_i) [D_i] \in H^2(E_X, \R) \text{,}
\]
where $[D_i]$ is the class dual to $D_i \subseteq E_X$ (oriented using omniorientation of $X_{\Sigma,\Lambda}$) and $h_\Delta$ is the support function of $\Delta$.
Further, we introduce
\[
  F_\gamma \colon \Pm_{\Sigma,\Lambda} \to \R; F_\gamma(\Delta) = \langle \rho(\Delta)^{n+i} \cdot p^*(\gamma),[E_X] \rangle \text{.}
\]
We noticed that $F_\gamma$ is a homogeneous polynomial of degree $n+i$ on $\Pm_{\Sigma,\Lambda}$. 

Recall that for any character $\lambda \in M$ we have an associated complex line bundle $\mathcal{L}_\lambda$ on $B$.
Taking Chern classes and extending by linearity, we obtain a morphism of vector spaces $c \colon M_\R \to H^2(B,\R)$.
Let $f_\gamma \colon M_\R \to \R$ be the function $f_\gamma(x) = \langle  c(x)^i \cdot \gamma, [B] \rangle$.
We defined a map $I_\gamma \colon \Pm_{\Sigma,\Lambda}\to \R$ to be 
\[
  I_\gamma(\Delta) = \int_\Delta f_\gamma(x) \diff\mu\text{.}
\]
By Theorem~\ref{ultimatepoly},  $I_\gamma$ is a homogeneous polynomial of degree $n+i$ on $\Pm_{\Sigma,\Lambda}$.

We prove Theorem~\ref{BKK} by induction on the parameter $0 \le i \le \frac{k}{2}$.

Let us start with the base case $i=0$.
If $i=0$, then $\gamma \in H^k(B, \R)$ is a multiple of the class dual to a point.
For simplicity, let us assume that this multiple is $1$.
So $p^*(\gamma)$ is the class dual to a fiber, which is a generilized quasitoric manifold, and hence $\rho(\Delta)^n\cdot p^*(\gamma)$ coincides with the self-intersection index of the codimesion 2 submanifold in $X$ corresponding to $\Delta$.
By the Theorem~\ref{thm:MABKK}, this can be computed by the normalized volume of $\Delta$ which equals to $n! \cdot I_\gamma(\Delta)$.
In other words, for $i=0$, Theorem~\ref{BKK} reduces to Theorem~\ref{thm:MABKK}.

As induction hypothesis, suppose that we know Theorem~\ref{BKK} for some $i-1\ge0$.
The induction step consists of proving that Theorem~\ref{BKK} is also true for $i$.
Since both $F_\gamma$ and $I_\gamma$ are homogeneous polynomials of the same degree $n+i$, in order to show equality between $(n+i)!\cdot I_\gamma(\Delta)$ and $i!\cdot F_\gamma(\Delta)$, it suffices to show that all their partial derivatives of order $n$ coincide.
In other words, it suffices to consider differential monomials $\partial_{i_1}^{k_1}\dots\partial_{i_r}^{k_r}$ where  $\partial_i= \partial/\partial_{h_i}$ are the partial derivatives along the coordinate vectors of $\Pm_{\Sigma,\Lambda} \simeq \R^s$ and $\sum_{i=1}^r k_i =n$.
Let us call the number $\sum_{i=1}^r (k_i-1)$ the \emph{multiplicity of the monomial} $\partial_{i_1}^{k_1}\dots\partial_{i_r}^{k_r}$.
In particular, a monomial has multiplicity $0$ if and only if it is square free.

The proof of equality of the partial derivatives of order $n$ of $(n+i)! \cdot I_\gamma(\Delta)$ and $i! \cdot F_\gamma(\Delta)$ is by induction on the multiplicity $m$ of the applied differential monomial.
We will refer to the induction over $i$ as the ``outer induction'' and we are going to call the induction over $m$ the ``inner induction''.

The base case of the inner induction (i.e., the case of square free differential monomials) is treated in the subsections.
The result of these calculations is summarized in Corollary~\ref{sqfree}.

%%%%%%%%%%%%%%%%%%%%%%%%%%%%%%%%%%%%%%%%%%%%%%%%%%%%%%%%%%%%%%%%%

\subsection{Differentiation with respect to square free monomials}

The results of differentiation of $I_\gamma$ with respect to square free differential monomials are given in Lemma~\ref{Ider}. Now we verify the base case of square free differential monomials for $F_\gamma$:

\begin{lemma}
  \label{Fder}
  Let $I = \{ i_1, \ldots, i_r\} \subseteq \{ 1, \ldots, s \}$ be a subset such that $\rho_{i_1},\ldots,\rho_{i_r}$ do not span a cone in $\Sigma$ and $k_1, \ldots, k_r$ are positive integers.
  If $\Delta \in \Pm_{\Sigma,\Lambda}$ is a multi-polytope, then we have
  \[
    \partial_{i_1}^{k_1} \cdots \partial_{i_r}^{k_r} \left(F_\gamma|_{\Pm_{\Sigma,\Lambda}}\right)(\Delta) = 0 \text{.}
  \]
  If $r = n = \dim(N_\R)$ and $\rho_{i_1}, \ldots, \rho_{i_n}$ span a cone in $\Sigma$ dual to the vertex $A\in M_\R$, we have 
  \[
    \partial_I \left(F_\gamma|_{\Pm_{\Sigma,\Lambda}}\right)(\Delta) = \s(I) \frac{(n+i)!}{i!}f_\gamma(A)\text{.}
  \]
\end{lemma}
\begin{proof}
  Without loss of generality we may assume that $i_j = j$ for $j = 1, \ldots, r$. For a monomial $\partial_1^{k_1} \cdots\partial_r^{k_r}$, let $\partial_I = \partial_1 \cdots \partial_r$ be the corresponding square free monomial.
  It is enough to show that $\partial_I F_\gamma(\Delta)  = 0$ in order  to prove that $\partial_1^{k_1} \cdots \partial_r^{k_r} F_\gamma(\Delta)  = 0$.

  We compute the expansion of the polynomial $F_\gamma$ at $\Delta$.
  This amounts to expressing $F_\gamma(\Delta+ \sum_{i=1}^s \lambda_i h_i)$ in terms of the monomials $\lambda_1^{\alpha_1} \cdots \lambda_s^{\alpha_s}$ for non-negative integers $\alpha_1, \ldots, \alpha_s$.
  A straightforward computation yields:
  \begin{align*}
    F_\gamma\left( \Delta+\sum_{i=1}^s\lambda_ih_i \right) &= \left( \rho(\Delta) + \sum_{i=1}^s \lambda_i [D_i] \right)^{n+i} \cdot p^*(\gamma)\\
    &= \sum_{\alpha_0 + \alpha_1 + \ldots + \alpha_s = n+i} \binom{n+i}{\alpha_0, \alpha_1, \ldots, \alpha_s} \cdot \rho(\Delta)^{\alpha_0} \cdot [D_1]^{\alpha_1} \cdots [D_s]^{\alpha_s} \cdot p^*(\gamma) \cdot \lambda_1^{\alpha_i} \cdots \lambda_s^{\alpha_s}\\
    &= \frac{(n+i)!}{(n+i-r)!} \cdot \rho(\Delta)^{n+i-r} \cdot [D_1] \cdots [D_r] \cdot p^*(\gamma) \cdot \lambda_1 \cdots \lambda_r + \text{(other terms)}\\
  \end{align*}
  where $\binom{n+i}{\alpha_0,\alpha_1,\ldots,\alpha_s}=\frac{(n+i)!}{\alpha_0! \cdot \alpha_1! \cdots \alpha_s!}$ denotes the usual multinomial coefficient.
  The derivative $\partial_I F_\gamma(\Delta)$ is equal to the coefficient in front of the monomial $\lambda_1 \cdots \lambda_r$ in the expression $F_\gamma(\Delta + \sum_{i=1}^s \lambda_i h_i)$:
  \[
    \partial_I F_\gamma(\Delta)= \frac{(n+i)!}{(n+i-r)!} \left\langle \rho(\Delta)^{n+i-r} \cdot [D_1] \cdots [D_r] \cdot p^*(\gamma), [E_X]\right\rangle\text{.}
  \]
  Characteristic submanifolds $D_1,\ldots, D_r$ intersect transversely in $E_X$, so the product $[D_1] \cdots [D_r]$ is the class Poincar\'{e} dual to their intersection.
  In the case that $e_1, \ldots, e_r$ do not generate a cone in $\Sigma$ the such intersection is empty, and so $\partial_I F_\gamma(\Delta)=0$. 

  For the proof of the second part, let $r = n = \dim(N_\R)$.
  We have
   \[
    \partial_I F_\gamma(\Delta)= \frac{(n+i)!}{i!} \left \langle \rho(\Delta)^{i} \cdot [D_1] \cdots [D_n] \cdot p^*(\gamma) , [E_X]\right\rangle\text{.}
  \]
 
 If $e_1, \ldots, e_n$ generate a cone in $\Sigma$ dual to the vertex $A$ of $\Delta$, then $[D_1] \cdots [D_n] = [E_A]$, where $E_A = E \times_T A$ is the torus invariant submanifold in $E_X$ corresponding to $A$. Note that the omniorientation of $X$ defines an orientation of $E_A$, hence the class $[E_A]$ is well-defined. In particular, the restriction of  the projection map $p\colon E_A \to B$ is a diffeomorphism which is orientation preserving if $\s(I)=1$ and orientaion reversing if $\s(I)=-1$.

Now, let $\widetilde \Delta =\Delta - A$ be a multi-polytope which is the translation of the multi-polytope $\Delta$ for which the vertex $A$ is at the origin.
  Since the vertex of $\widetilde \Delta$ corresponding to $A$ is at the origin, we get
  \[
    h_{\widetilde \Delta}(e_{1}) = \ldots = h_{\widetilde \Delta}(e_{n}) = 0, \qquad \text{and so} \qquad \rho(\widetilde \Delta) = \sum_{j>n}h_{\widetilde \Delta}(e_j) \cdot [D_j].
  \]
  Hence, by the first part, we get $\rho(\widetilde \Delta) \cdot [D_1] \cdots [D_n] = 0$ as there is no cone in $\Sigma$ with more than $n$ rays.
  Therefore:
  \[
    \rho(\Delta)^i \cdot [D_1] \cdots [D_n]  \cdot p^*(\gamma)= \rho(\widetilde \Delta + A)^i \cdot [D_1] \cdots [D_n] \cdot p^*(\gamma)= \rho(A)^i  \cdot [E_A] \cdot p^*(\gamma) \text{.}
  \]
  By Proposition~\ref{cherneq}, $\rho(A)=p^* c(A)$.
  Since $p\colon E_A\to B$ is a diffeomorphism, we get:
  \[
    \left\langle\rho(A)^i  \cdot p^*(\gamma)  \cdot [E_A]  , [E_X]\right\rangle = \left\langle(p^* c(A))^i\cdot p^*(\gamma)  \cdot [E_A] , [E_X]\right\rangle = \left\langle c(A)^i\cdot \gamma , \s(I)[B]\right\rangle= \s(I)f_\gamma(A), 
  \]
  and therefore $\partial_I F_\gamma(\Delta) = \s(I)\frac{(n+i)!}{i!}f_\gamma(A)$.
\end{proof}

The base case of the inner induction is an immediate corollary of Lemmas~\ref{Ider} and~\ref{Fder}:
\begin{corollary}[Base case of the inner induction]
  \label{sqfree} 
  For any $i\le \frac{k}{2}$, any $\gamma \in H^{k-2i}(B,\R)$ and any square free differential monomial $\partial_I=\partial_{i_1}\dots\partial_{i_n}$ of order $n$ (where $I = \{ i_1, \ldots, i_n\} \subseteq \{1, \ldots, s\}$), we have:
  \[
    \partial_I \left( (n+i)!\cdot I_\gamma(\Delta)\right)= \partial_I \left( i!\cdot F_\gamma(\Delta)\right) \text{.}
  \]
\end{corollary}

%%%%%%%%%%%%%%%%%%%%%%%%%%%%%%%%%%%%%%%%%%%%%%%%%%%%%%%%%%%%%%%%%

\subsection{ The inner induction step}
In this subsection, the index set $I \subseteq \{1, \ldots, s \}$ is considered as being a multiset.
As the induction hypothesis, suppose that $\partial_I \left( (n+i)!\cdot I_\gamma(\Delta)\right) = \partial_I \left( i!\cdot F_\gamma(\Delta) \right)$, for each differential monomial $\partial_I$ of multiplicity $m-1\ge0$.
It remains to show that the equality is true for differential monomials of multiplicity $m$.
As before, to keep notation simple assume $i_j=j$ for $j =1, \ldots, r$, so that $\partial_I=\partial_{1}^{k_1}\ldots\partial_{r}^{k_r}$ for positive integers $k_1, \ldots, k_r$.
By reordering the coordinates of $\Pm_{\Sigma,\Lambda}$, we may assume $k_1>1$. 
By Lemmas~\ref{Ider} and~\ref{Fder}, it suffices to consider the case where the vectors $e_1, \ldots, e_r$ form a cone in $\Sigma$ (as otherwise $\partial_I I_\gamma(\Delta) = 0 = \partial_I F_\gamma(\Delta)$).

We are going to express $\partial_1$ as a combination of Lie derivative $L_v$ (for some $v \in M_\R$) and other partial derivatives.
Then the (inner) induction step will follow by an explicit computation of $L_vI_\gamma(\Delta)$ and $L_vF_\gamma(\Delta)$.

As $e_1, \ldots, e_r$ form a cone in the smooth fan $\Sigma$, they can be completed to a basis of $N_\R$.
We take $v$ to be the first vector of the dual basis of $M_\R$.
Then $\langle v, e_1 \rangle = 1$, and  $\langle v, e_j \rangle = 0$ for $j = 2, \ldots, r$.
Since $v\in M_\R$ (considered as an element of $\Pm_{\Sigma,\Lambda}$) is given by $v = \sum_{i=1}^s \langle v, e_i \rangle h_i$, we obtain that $L_v = \sum_{i=1}^s \langle v,e_i \rangle \partial_i$, and thus $\partial_1 = L_v - \sum_{j>r}\langle v, e_j \rangle \partial_j$.
We get:
\begin{align*}
  \partial_I = \partial_1^{k_1} \ldots \partial_r^{k_r} &= \left( L_v - \sum_{j>r}\langle v, e_j \rangle \partial_j \right) \partial_1^{k_1-1} \cdot \partial_2^{k_2} \cdots \partial_r^{k_r} \\
  &= L_v \cdot \partial_1^{k_1-1} \cdot \partial_2^{k_2} \cdots \partial_r^{k_r} - \sum_{j>r}\langle v, e_j \rangle \cdot \partial_1^{k_1-1} \cdot \partial_2^{k_2} \cdots \partial_r^{k_r} \cdot \partial_j \text{.}
\end{align*}
Since $k_1>1$ and $j>r$, each monomial in the sum $\sum_{j>r}\langle v, e_j \rangle \partial_1^{k_1-1} \cdot \partial_2^{k_2} \cdots \partial_r^{k_r} \cdot \partial_j$ has multiplicity $m-1$, so by the  inner induction hypothesis, we get
\[
  (n+i)! \cdot \left( \sum_{j>r} \langle v, e_j \rangle \partial_1^{k_1-1} \cdot \partial_2^{k_2} \cdots \partial_r^{k_r} \cdot \partial_j \right) \cdot I_\gamma(\Delta) = i! \cdot \left( \sum_{j>r} \langle v, e_j \rangle \partial_1^{k_1-1} \cdot \partial_2^{k_2} \cdots \partial_r^{k_r} \cdot \partial_j \right) \cdot F_\gamma(\Delta) \text{.}
\]
It remains to consider the first summand:
\begin{lemma}
  \label{Lvind}
  In the situation above, we have $(n+i)! \cdot L_v I_\gamma(\Delta) = i! \cdot L_v F_\gamma(\Delta)$.
\end{lemma}
\begin{proof}
  A direct calculation shows:
  \begin{align*}
    L_v  I_\gamma(\Delta) &= \partial_t \Bigg|_{t=0} \left(  \int_{\Delta+tv} c(x)^{i} \cdot \gamma \diff \mu \right) =  \partial_t\Bigg|_{t=0} \left(  \int_{\Delta} c(x+tv)^{i} \cdot \gamma \diff \mu \right) =\\
    &= \partial_t \Bigg|_{t=0}\left(  \int_{\Delta} \left(\sum_{a=0}^i \binom{i}{a} t^a\cdot c(v)^{a}c(x)^{i-a} \right) \cdot \gamma \diff \mu \right) =  \int_{\Delta}  i\cdot   c(v)c(x)^{i-1}  \cdot \gamma \diff \mu\\
    &= i\cdot\int_{\Delta} c(x)^{i-1}  \cdot  (c(v) \cdot \gamma) \diff \mu = i\cdot I_{c(v)\cdot \gamma}(\Delta) \text{.}
  \end{align*}
  Similarly, by a direct calculation and Proposition~\ref{cherneq}, we get
  \begin{align*}
    L_v F_\gamma(\Delta) &= \partial_t \Big|_{t=0} \Big(\rho(\Delta+tv)^{n+i}\cdot p^*(\gamma)  \Big)=
\partial_t \Big|_{t=0} \Big((\rho(\Delta)+t\rho(v))^{n+i}\cdot p^*(\gamma)  \Big)=\\
    &= \partial_t \Bigg|_{t=0} \left( \left( \sum_{a=0}^{n+i} \binom{n+i}{a} t^a \cdot \rho(v)^a \cdot \rho(\Delta)^{n+i-a}  \right)\cdot p^*(\gamma) \right) = (n+i)\cdot \rho(v) \cdot \rho(\Delta)^{n+i-1} \cdot p^*(\gamma)\\
    &= (n+i)\cdot  \rho(\Delta)^{n+i-1} \cdot p^*(c(v) \cdot \gamma) = (n+i)\cdot  F_{c(v)\gamma}(\Delta).
  \end{align*}
  Since $c(v) \cdot \gamma \in H^{k-2(i-1)}(B,\R)$, by the induction hypothesis of the outer induction, we have
  \[
    (n+i-1)! \cdot I_{c(v) \cdot \gamma}(\Delta) = (i-1)!\cdot F_{c(v) \cdot \gamma}(\Delta).
  \]
  Therefore,
  \[
    (n+i)! \cdot L_v I_\gamma(\Delta) = (n+i)! \cdot  i \cdot I_{c(v) \cdot \gamma}(\Delta) = i! \cdot (n+i) \cdot F_{c(v) \cdot \gamma}(\Delta) =  i! \cdot  L_v F_\gamma(\Delta) \text{.} \qedhere
  \]
\end{proof}

%%%%%%%%%%%%%%%%%%%%%%%%%%%%%%%%%%%%%%%%%%%%%%%%%%%%%%%%%%%%%%%%%
%%%%%%%%%%%%%%%%%%%%%%%%%%%%%%%%%%%%%%%%%%%%%%%%%%%%%%%%%%%%%%%%%

\section{\texorpdfstring{Graded-commutative $n$-self-dual  algebras}{Graded-commutative n-self-dual algebras}}
\label{commalg}

In this section we briefly summarize some algebraic results obtained in the recent papers~\cite{hof2020} and~\cite{KhoM} that we are going to rely on when describing cohomology rings of quasitoric bundles in a way similar to the classical approach by Khovanskii and Pukhlikov~\cite{PK92}.

%%%%%%%%%%%%%%%%%%%%%%%%%%%%%%%%%%%%%%%%%%%%%%%%%%%%%%%%%%%%%%%%%%
\subsection{The general case}
%%%%%%%%%%%%%%%%%%%%%%%%%%%%%%%%%%%%%%%%%%%%%%%%%%%%%%%%%%%%%%%%%%

Here, we introduce all the necessary algebraic preliminaries that will be used to describe the entire cohomology rings of quasitoric bundles. Our presentation follows that in~\cite{hof2020}.

Recall that a graded $\K$-algebra $A = A^0 \oplus A^1 \oplus \cdots \oplus A^k \oplus \cdots$ over a field $\K$ is called \emph{graded-commutative} if for any pair of homogeneous elements $x$, $y$, the following relation holds:
  \[
    xy = (-1)^{\deg(x)\deg(y)} yx
  \]
where $\deg(x)$ (resp.~$\deg(y)$) denotes the degree of $x$ (resp.~of~$y$).

\begin{definition}  
  Let  $n$ be a natural number. A graded-commutative $\K$-algebra $A$ is said to be \emph{Poincar\'{e} $n$-self-dual} (or just, \emph{$n$-self-dual}) if the following conditions hold:
  \begin{enumerate}
  \item The algebra $A$ has a multiplicative unit element $e$ which is homogeneous of degree zero, i.e. $e \in A^0 \subseteq A$.
  \item The homogeneous components $A^k$ for $k>n$ vanish, i.e. $A^k=0$ for $k>n$, and $\dim_{\K}(A^n) = 1$.
  \item The pairing $A^k \times A^{n-k} \rightarrow A^n$ induced by multiplication in the algebra $A$ is non-degenerate for $0 \le k \le n$.
  \end{enumerate}
\end{definition}

Here is a key example we are most interested in throughout this paper.

\begin{example}
  \label{ex2}
  Let $M$ be a connected compact oriented $n$-dimensional manifold.
  By Poincar\'{e} duality, the cohomology ring $A = H^*(M, \K)$ is an $n$-self-dual graded-commutative $\K$-algebra.
  Moreover, $A$ is equipped with two extra structures:
  \begin{enumerate}
  \item A $\K$-linear function $\ell_* \colon A \to \K$  given as follows.
    For $\alpha\in A^n$, we let $\ell_*(\alpha)$ be equal to the value of the cohomology class $\alpha$ on the fundamental class of the oriented manifold $M$.
    For $\alpha\in A^k, k\neq n$, we set $\ell_*(\alpha)=0$. We extend $\ell_*$ to the whole $A$ by linearity.
  \item A $\K$-bilinear intersection form $F_{\ell_*}$ on $A$ defined by the identity $F_{\ell_*}(\alpha, \beta) = \ell_*(\alpha \cdot \beta)$ (recall that we denote the cup product in the cohomology ring $A$ by ``$\cdot$'').
\end{enumerate}
\end{example}

\begin{remark}
Note that $F_{\ell_*}$ is non-degenerate; in fact, the non-degeneracy of $F_{\ell_*}$ is equivalent to the Poincar\'{e} duality on $A$. 
\end{remark}

Let $ B = B^0 \oplus B^1 \oplus \cdots \oplus B^k \oplus \cdots$ be a graded-commutative $\K$-algebra with multiplicative unit element $e \in B^0$ and $\dim_{\K}(B^0) = 1$. Our goal is to describe all $n$-self dual factor-algebras of $B$.
To that extent we introduce the following notion.

\begin{definition} 
  An ideal $I\subseteq B$  is called \emph{$n$-self-dual} (or just, \emph{$n$-sd ideal}) if $I$ is a two-sided homogeneous ideal in $B$ and the factor-algebra $A = B/I$ is $n$-self-dual.
\end{definition}

It's easy to see that if $I\subseteq B$ is an $n$-sd ideal, then $L=I\cap B^n$ is a hyperplane in $B^n$. As hyperplanes are in correspondence with linear functions, we introduce the next definition.

\begin{definition}  
  A  linear function $\ell \colon B\to \K$ on a graded-commutative $\K$-algebra $B$ is \emph{$n$-homogeneous} if it is not identically equal to zero on $B$, but for any $m \ne n$ its restriction to the homogeneous component $B^m$ vanishes.
\end{definition}

Denote by $L_\ell\subseteq B^n$ the hyperplane in $B^n$ defined by the identity 
$$
L_\ell: = \{\ell = 0\} \cap B^n.
$$
Observe that any hyperplane $L \subseteq B^n$ is equal to a hyperplane of the type $L_\ell$ for a unique (up to a non-zero scalar multiple) $n$-homogeneous linear function $\ell$.

In the rest of this subsection, we explain how to obtain an $n$-self-dual factor-algebra from an $n$-homogeneous linear function $\ell \colon B \to \K$.

\begin{definition}
  \label{def4} 
  For an $n$-homogeneous linear function $\ell \colon B \to \K$ let $I_1(L_\ell)$ and $I_2(L_\ell)$ be the subsets of $B$ defined by the following conditions:
  \begin{enumerate}
  \item an element $a \in B$ belongs to $I_1(L_\ell)$ if and only if $\ell(ab) = 0$ for all $b \in B$;
  \item an element $b \in B$ belongs to $I_2(L_\ell)$ if and only if $\ell(ab) = 0$ for all $a \in B$.
  \end{enumerate}
\end{definition}

For any hyperplane $L_\ell\subseteq B^n$ one has:
$$
I_1(L_\ell)=I_2(L_\ell)=:I(L_\ell)\text{ is an }n-\text{sd ideal in }B,
$$
where $\ell$ is an $n$-homogeneous linear function. Note that the ideals $I_1(L_\ell)$ and $I_2(L_\ell)$ depend only on the hyperplane $L: = L_\ell$, and so the notation $I(L):=I(L_\ell)$ makes sense.
 
\begin{definition}
  Let $\ell \colon B \to \K$ be an $n$-homogeneous linear function on a graded-commutative algebra $B$.
  The \emph{Frobenius bilinear form} $F_\ell \colon B \times B \to \K$ associated with $\ell$ is the form defined by the identity $F_\ell(a,b) = \ell(a \cdot b)$.
\end{definition}

It is easy to see that the $n$-sd ideal $I(L)$ coincides with the left and also the right kernel of $F_\ell$ for $L=L_\ell$. 

It turns out that the set of $n$-sd factor-algebras $A=B/I$ equipped with non-degenerate intersection forms is in one-to-one correspondence with the set of $n$-homogeneous linear functions on $B$. More precisely, the following fundamental result takes place.

\begin{theorem}
  \label{thm-nsdquatient}
  Let $B$ be a graded-commutative algebra over a field $\K$, and let $A$ be an $n$-sd factor-algebra of $B$ with a chosen isomorphism $\phi\colon A^n\to \K$.
  Let $\rho \colon B \to A$ be the natural epimorphism of algebras.
  Then
  \[
    A\simeq B/I(L_\ell),
  \]
  where $\ell$ is the $n$-homogeneous linear function on $B$ defined by
  \[
    \ell\colon B^n\to \K, \quad \ell(b)=\phi(\rho(b)),
  \]
  and extended by 0 to $B^i$ with $i\ne n$.

  Moreover, for any $n$-homogeneous linear function $\ell$, the algebra $A=B/I(L_\ell)$ is $n$-self dual and the Frobenius form $F_{\ell}$ induces a non-degenerate pairing on $A$.
\end{theorem}

%%%%%%%%%%%%%%%%%%%%%%%%%%%%%%%%%%%%%%%%%%%%%%%%%%%%%%%%%%%%%%%%%%%%%

%%%%%%%%%%%%%%%%%%%%%%%%%%%%%%%%%%%%%%%%%%%%%%%%%%%%%%%%%%%%%%%% 
\subsection{The even-dimensional case}
%%%%%%%%%%%%%%%%%%%%%%%%%%%%%%%%%%%%%%%%%%%%%%%%%%%%%%%%%%%%%%%%

In this subsection, we are going to introduce the algebraic results we need in order to obtain a more convenient way to describe the subring $A^*(E_X)$ of all even-dimensional classes in the cohomology ring $H^*(E_X)$ of a quasitoric bundle $E_X$ with an even-dimensional base $B$. Our presentation follows that in~\cite{KhoM}.

Let $\K$ be a field of zero characteristic, in this paper we assume $\K=\Q,\R$, or $\C$. Throughout this subsection $A$ will be a $2n$-self-dual graded-commutative algebra over $\K$ with vanishing odd-degree part (therefore, $A$ is a commutative algebra in the usual sense). 

Let $W\subset A$ be a vector subspace, which multiplicatively generates $A$. 
Then $W$ has a natural graded vector space structure: $W=\oplus_{k=0}^{n}W_{2k}$, where $W_{2k}:=W\cap A^{2k}$.

Let $\pi\colon V\to \oplus_{k=1}^{n}W_{2k}$ be a surjective linear mapping for a finite-dimensional vector space $V$. This makes $V$ into a graded vector space via $V=\oplus_{k=1}^{n}V_{2k}$, where $V_{2k}:=\pi^{-1}(W_{2k})$.

Observe that $\pi$ extends to an epimorphism of algebras
$$
\pi\colon \Sym(V)\to A.
$$

Furthermore, the grading on $V$ gives rise to a grading on $\Sym(V)$:
$$
\deg(v_1^{i_1}\ldots v_{n}^{i_n})=2\sum\limits_{k=1}^{n}ki_k,\text{ for any }v_k\in V_{2k},
$$
making $\pi$ into a graded morphism.

Firstly, we are going to describe the ideal $I_\ell\subset\Sym(V)$ such that
$$
A\simeq\Sym(V)/I_\ell.
$$
Here, $\ell$ is defined on $\Sym(V)$ as a composition of maps:
$$
\ell=\alpha\circ\pi\colon\Sym(V)\to\K,
$$
where $\alpha\colon A\to\K$ is a $2n$-homogeneous linear function such that $\alpha\colon A^{2n}\simeq\K$. 
Note that $\ell$ is a $2n$-quasi-homogeneous linear function.

Let us denote by $I_\ell$ the following ideal in $\Sym(V)$:
$$
I_\ell=\{P\in\Sym(V)\,|\,\ell(P\cdot Q)=0\text{ for any }Q\in\Sym(V)\}.
$$

By Theorem~\ref{thm-nsdquatient} (see also \cite{kazarnovskii2021newton} for similar results), the kernel of the graded epimorphism of algebras $\pi\colon\Sym(V)\to A$ is given by the ideal just defined: $\mathrm{ker}(\pi)=I_\ell$, and hence
$$
A\simeq\Sym(V)/I_\ell.
$$
Below we describe the ideal $I_\ell$ in an alternative  way, which is more useful in some applications.

Observe that the symmetric algebra $\Sym(V)$ can be identified with the space $\Diff(V)$ of differential operators with constant coefficients on $V$ as follows. The zero degree component $\Sym_0(V)\simeq\K$ corresponds to operators of multiplication by a number. For positive degree elements one can define a map
$$
\mathcal D\colon v_1\ldots v_k\mapsto L_{v_1}\ldots L_{v_k},
$$
where $L_v$ is the Gateaux derivative in the direction $v\in V$. For an element $P\in\Sym(V)$ we denote by $\mathcal D_P\in\Diff(V)$ the corresponding differential operator.

Define a map
$$
\Exp^*\colon\Sym(V)^{\vee}\to\K[[V]],\quad\Exp^*\colon\alpha\mapsto\sum\limits_{i=0}^{\infty}\alpha\left(\frac{x^i}{i!}\right).
$$
It turns out that $\Exp^*(\ell)$ is a quasi-homogeneous polynomial of degree $2n$ for the $2n$-quasi-homogeneous linear function $\ell$.

Now we are ready to give an alternative description for the ideal $I_\ell$ in $\Sym(V)$.

\begin{theorem}\label{potential}
Let $A$ be a $2n$-self-dual graded-commutative algebra over $\K$ with vanishing odd-degree part. Let $V$ be a finite-dimensional graded vector space over $\K$ and $\pi\colon\Sym(V)\to A$ be a graded epimorphism of algebras.

Suppose $\ell\colon \Sym(V)\to\K$ is a $2n$-quasi-homogeneous linear function. Then we have
$$
A\simeq\Sym(V)/\Ann(\Exp^*(\ell)),
$$
where 
$$
\Ann(\Exp^*(\ell)):=\{P\in\Sym(V)\,|\,\mathcal D_P(\Exp^*(\ell))=0\}.
$$
\end{theorem}

We call the quasi-homogeneous polynomial $\Exp^*(\ell)$ the \emph{potential} of $A$ and denote it by $P_A$.

%%%%%%%%%%%%%%%%%%%%%%%%%%%%%%%%%%%%%%%%%%%%%%%%%%%%%%%%%%%%%%%%%

\section{Applications}
\label{sec:applications}

%%%%%%%%%%%%%%%%%%%%%%%%%%%%%%%%%%%%%%%%%%%%%%%%%%%%%%%%%%%%%%%% 

\subsection{Pukhlikov-Khovanskii description of cohomology ring of a quasitoric bundle}\label{sec:cohtoricbundle}

In this section, we apply the results from Sections~\ref{BKKproof} and~\ref{commalg} to give a Pukhlikov-Khovanskii-type description of the cohomology ring of a quasitoric bundle that generalizes the descriptions obtained in~\cite{ayzenberg2016volume} for quasitoric manifolds and in~\cite{hof2020} for toric bundles.

Let, as before, $p\colon E\to B$ be a principal $T$-bundle, $\Sigma$ a complete simplicial fan with $\Sigma(1) = \{\rho_1, \ldots, \rho_r\}$, $X = X_{\Sigma,\Lambda}$ a generalized quasitoric manifold, and $E_X$ the corresponding quasitoric bundle.

By the Leray-Hirsch theorem (see Theorem~\ref{thm:Leray-Hirsch}) the cohomology ring $H^*(E_X,\R)$ is a quotient of the polynomial algebra $R[x_1,\ldots,x_r]$, where $R=H^*(B,\R)$ is the cohomology ring of the base.
\begin{theorem}
  \label{cohbundle}
  In the notation from above, the cohomology ring of $E_X$ is given by
  \[
    H^*(E_X,\R) \simeq R[x_1,\ldots,x_r]/I(L_\ell)
  \]
 where $\ell \colon R[x_1,\ldots,x_r]\to \R$ is a $(k+2n)$-homogeneous linear function defined by: 
  \[
    \ell(\gamma \cdot x_{i_1} \cdots x_{i_s}) = I_\gamma(\rho_{i_1},\ldots, \rho_{i_s})
  \]
  for any monomial $\gamma\cdot x_{i_1}\cdots x_{i_s}$ with $\deg(\gamma)+2s=k+2n$.
\end{theorem}
\begin{proof}
  By Theorem~\ref{thm-nsdquatient}, as $H^*(E_X,\R)$ is a graded commutative $(k+2n)$-self dual factor algebra of $R[x_1,\ldots,x_r]$, it is given by $R[x_1,\ldots,x_r]/I(L_\ell)$ for a $(k+2n)$-homogeneous linear function $\ell$ which is obtained by pairing cohomology classes with the fundamental class of $E_X$.
  Hence the statement follows by Theorem~\ref{BKK}.
\end{proof}

For the next two theorems, let us assume that $\dim(B)=2k$ is even. In this case, $\dim E_X = 2(k+n)$ is also even. Then let us denote by $A^*(B), A^*(E_X)$ the corresponding rings of even cohomology classes: 
\[
A^*(B) = \bigoplus_{i=0}^{2k} H^{2i}(B,\R), \quad A^*(E_X) =  \bigoplus_{i=0}^{2(k+n)} H^{2i}(E_X,\R).
\]
Rings $A^*(B), A^*(E_X)$ are commutative graded rings satisfying Poincar\'e duality. Hence we can use Theorem~\ref{potential} to obtain their presentations.

Denote by $P_B$ the potential of $A^*(B)$ with respect to the generating vector space $A^*(B)$ and by $P_{E_X}$ the potential of $A^*(E_{X})$ with respect to $A^*(B) \oplus \Pm_{\Sigma,\Lambda}$. The following theorem provides a relation between potentials $P_{E_X}$ and $P_B$.

\begin{theorem}\label{thm:torpot}
  For any $\Delta \in \Pm_{\Sigma,\Lambda}$ and $\gamma\in A^*(B)$ we have
  \[
P_{E_{X}}(\gamma,\Delta) = \int_\Delta  P_B(c(\lambda)+\gamma) \diff \lambda.
  \]
\end{theorem}

\begin{proof}
The proof relies on Theorem~\ref{BKK1} and will be obtain by a direct calculation. 
Recall that by Theorem~\ref{potential}, the potentials $P_{E_X}, P_B$ are given by
\[
P_{E_X}= \Exp^*\langle [E_X], \cdot \rangle,\quad P_B= \Exp^*\langle [B], \cdot \rangle
\]
In other words, the evaluations of potentials $P_{E_X}$ and $P_B$ on elements of $A^*(B) \oplus \Pm_{\Sigma,\Lambda}$ and $A^*(B)$ respectively are given by the following formulas 
\[
P_{E_X}(\gamma,\Delta) = \left\langle [E_X],\sum_{j=0}^{\infty} \frac{(p^*(\gamma)+\rho(\Delta))^{j}}{j!}\right\rangle, \quad P_B(\gamma) =  \left\langle [B],\sum_{j=0}^{\infty} \frac{\gamma^{j}}{j!}\right\rangle
\]
By definition of horizontal parts $h_i(\Delta)$ of $\Delta$ we have
  \[
P_{E_X}(\gamma,\Delta) =   \left\langle [E_X], p^*(\exp(\gamma)) \cdot \sum_{j=0}^{\infty} \frac{\rho(\Delta)^{j}}{j!}\right\rangle =  \left\langle [B],\exp(\gamma) \cdot \sum_{j=0}^{\infty} \frac{h_j(\Delta)}{j!}\right\rangle.
  \]
Now, by Theorem~\ref{BKK1} we get $h_j(\Delta) = 0$ for $j<n$ and
 \[
    h_{n+i}(\Delta) = \frac{(n+i)!}{i!} \int_\Delta c(\lambda)^i \diff \lambda \text{,}
 \]
otherwise. Hence we obtain
\begin{equation*}
\begin{split}
P_{E_X}(\gamma,\Delta) = \left\langle [B],\exp(\gamma) \cdot \sum_{j=0}^{\infty} \frac{h_j(\Delta)}{j!}\right\rangle &= \left\langle [B],\exp(\gamma) \cdot \sum_{i=0}^{\infty} \frac{h_{n+i}(\Delta)}{(n+i)!}\right\rangle = \\
\left\langle [B],\exp(\gamma) \cdot \sum_{i=0}^{\infty} \frac{1}{i!} \int_\Delta c(\lambda)^i \diff \lambda \right\rangle &= \left\langle [B],\int_\Delta\exp(\gamma+c(\lambda)) \diff \lambda\right\rangle = \\
\int_\Delta \langle [B],\exp(c(\lambda)+\gamma)\rangle \diff \lambda &= \int_\Delta  P_B(c(\lambda)+\gamma) \diff \lambda.
\end{split}    
\end{equation*}
This finishes the proof.
\end{proof}

The next result is a direct generalization of~\cite[Theorem 5.8]{KhoM} to quasitoric bundles.

\begin{theorem}\label{subringevendim}
In the above notation, suppose $\dim B$ is even. Then the ring of even-degree cohomology classes of a quasitoric bundle is given by
$$
A^*(E_X)=\Sym(A^*(B)\oplus\Pm_{\Sigma,\Lambda})/\Ann(P_{E_X}),
$$
where the quasi-homogeneous polynomial $P_{E_X}$ is given by
$$
P_{E_{X}}(\gamma,\Delta) = \int_\Delta  P_B(c(\lambda)+\gamma) \diff \lambda.
$$
\end{theorem}
\begin{proof}
The proof goes by a direct application of Theorem~\ref{potential} and Theorem~\ref{thm:torpot}.  
\end{proof}

%%%%%%%%%%%%%%%%%%%%%%%%%%%%%%%%%%%%%%%%%%%%%%%%%%%%%%%%%%%%%%%%%%%%%%%%%%%%%%%%%%%%%%%%%%%%%%%%%%%%%%%%%%%%%%%%%%%%%%%%%%

\subsection{Stanley-Reisner description of cohomology ring of a quasitoric bundle}

In this subsection we provide a generalization of the Stanley-Reisner description for the cohomology ring to quasitoric bundle case. A  version of Theorem~\ref{thm:SR} which works for more general torus manifold bundles was obtained in \cite{dasgupta2019cohomology}.

\begin{theorem}
  \label{thm:SR}
  Let $E_X$ be a quasitoric bundle associated with a generalized quasitoric manifold $X=X_{\Sigma,\Lambda}$ as before. 
  Then $H^*(E_X,\R)$ is isomorphic (as an $H^*(B, \R)$-algebra) to the quotient of $H^*(B,\R)[x_1, \ldots, x_r]$ by
  \[
     \rleft( \left\langle x_{j_1} \cdots x_{j_k} \colon \rho_{j_1}, \ldots, \rho_{j_k} \; \text{do not span a cone of} \; \Sigma \right\rangle + \left\langle c\left( \lambda \right) - \sum_{i=1}^n \langle e_i, \lambda \rangle x_i \colon \lambda \in M \right\rangle \rright) \text{.}
  \]
\end{theorem}
\begin{proof}
The proof is identical to the proof of \cite[Theorem 3.5]{hof2020}.
\end{proof}

Note the similarities with the Stanley-Reisner description of the cohomology ring of toric varieties and quasitoric manifolds.
Indeed, the first ideal in Theorem~\ref{thm:SR} corresponds to the Stanley-Reisner ideal of the corresponding toric variety.

%%%%%%%%%%%%%%%%%%%%%%%%%%%%%%%%%%%%%%%%%%%%%%%%%%%%%%%%%%%%%%%% 
%%%%%%%%%%%%%%%%%%%%%%%%%%%%%%%%%%%%%%%%%%%%%%%%%%%%%%%%%%%%%%%%%

\subsection{Brion description of cohomology ring of a quasitoric bundle}

In this subsection we give a proof of a generalization of the Brion description for a cohomology ring to the quasitoric bundle case.

\begin{theorem}
  \label{thm:Br}
  Let $E_X$ be a quasitoric bundle associated with a generalized quasitoric manifold $X=X_{\Sigma,\Lambda}$ as before. 
  Then $H^*(E_X,\R)$ is isomorphic (as an $H^*(B,\R)$-algebra) to the quotient of $H^*(B,\R)\otimes\PP_{\Sigma,\Lambda}$ by the ideal
 $$
 \langle c\left( \lambda \right)\otimes 1 - 1\otimes\langle \cdot, \lambda \rangle \colon \lambda \in M\rangle,
 $$
 where $\langle\cdot,\lambda\rangle\in\PP^{1}_{\Sigma,\Lambda}$ is a ($\Lambda$--piecewise) linear function on $\Sigma$.
 \end{theorem}
\begin{proof}
For the generalized quasitoric manifold $X$, Theorem~\ref{cohgenqtrmfd} (1) and Theorem~\ref{thm:Brion} imply that
$$
\R[\Sigma]\simeq H^*_{T}(X;\R)\simeq \PP_{\Sigma,\Lambda}.
$$

It is easy to see that the composition isomorphism of $\R$-algebras is given by the mapping $v_i\mapsto \phi_i$, for all $1\leq i\leq s$.
Hence the set of piecewise linear functions $\{\phi_i\colon 1\leq i\leq s\}$ forms a basis of $\PP^{1}_{\Sigma,\Lambda}$. 

Then any $\Lambda$-piecewise linear function $\langle\cdot,\lambda\rangle$ on $\Sigma$ can be uniquely expressed as a linear combination of elements of this basis, namely
$$
\langle\cdot,\lambda\rangle=\sum_{i=1}^n \langle e_i, \lambda \rangle\phi_{i}. 
$$

It remains to apply the Stanley-Reisner description given by Theorem~\ref{thm:SR} proved in the previous subsection.
\end{proof}

%%%%%%%%%%%%%%%%%%%%%%%%%%%%%%%%%%%%%%%%%%%%%%%%%%%%%%%%%%%%%%%%%
%%%%%%%%%%%%%%%%%%%%%%%%%%%%%%%%%%%%%%%%%%%%%%%%%%%%%%%%%%%%%%%%%

\bibliographystyle{alpha}
\bibliography{tbc}

\newcommand{\etalchar}[1]{$^{#1}$}
\begin{thebibliography}{BEM{\etalchar{+}}17}

\bibitem[AM16]{ayzenberg2016volume}
Anton Ayzenberg and Mikiya Masuda.
\newblock Volume polynomials and duality algebras of multi-fans.
\newblock {\em Arnold Mathematical Journal}, 2(3):329--381, 2016.

\bibitem[BEM{\etalchar{+}}17]{Buchstaber2017CohomologicalRO}
Victor~M. Buchstaber, Nikolay~Yu. Erokhovets, Mikiya Masuda, Taras~E. Panov,
  and Seonjeong Park.
\newblock Cohomological rigidity of manifolds defined by 3-dimensional
  polytopes.
\newblock {\em Russian Mathematical Surveys}, 72:199--256, 2017.

\bibitem[Ber75]{Bernstein}
D.~N. Bernstein.
\newblock The number of roots of a system of equations.
\newblock {\em Funkcional. Anal. i Prilo\v{z}en.}, 9(3):1--4, 1975.

\bibitem[BGLV20]{Baralic2020ToricOA}
Djordje Barali{\'c}, Jelena Grbi{\'c}, Ivan~Yu. Limonchenko, and Aleksandar
  Vuci{\'c}.
\newblock Toric objects associated with the dodecahedron.
\newblock {\em Filomat}, 34:2329--2356, 2020.

\bibitem[BKH76]{BKK}
D.~N. Bernstein, A.~G. Ku\v{s}nirenko, and A.~G. Hovanski\u{\i}.
\newblock Newton polyhedra.
\newblock {\em Uspehi Mat. Nauk}, 31(3(189)):201--202, 1976.

\bibitem[BP02]{buchstaber2002torus}
Victor~M. Buchstaber and Taras~E. Panov.
\newblock {\em Torus actions and their applications in topology and
  combinatorics}.
\newblock Number~24. American Mathematical Soc., 2002.

\bibitem[BP16]{Buchstaber2016OnMD}
Victor~M. Buchstaber and Taras~E. Panov.
\newblock On manifolds defined by 4-colourings of simple 3-polytopes.
\newblock {\em Russian Mathematical Surveys}, 71:1137--1139, 2016.

\bibitem[BPR06]{Buchstaber2006SpacesOP}
Victor~M. Buchstaber, Taras~E. Panov, and Nigel Ray.
\newblock Spaces of polytopes and cobordism of quasitoric manifolds.
\newblock {\em Moscow Mathematical Journal}, 7:219--242, 2006.

\bibitem[CLS11]{tor-var}
D.~A. Cox, J.~B. Little, and H.~K. Schenck.
\newblock {\em Toric varieties}, volume 124 of {\em Graduate Studies in
  Mathematics}.
\newblock American Mathematical Society, Providence, RI, 2011.

\bibitem[CMS08]{Choi2008QuasitoricMO}
Suyoung Choi, Mikiya Masuda, and Dong~Youp Suh.
\newblock Quasitoric manifolds over a product of simplices.
\newblock {\em arXiv: Algebraic Topology}, 2008.

\bibitem[DJ91]{davis1991convex}
Michael~W. Davis and Tadeusz Januszkiewicz.
\newblock Convex polytopes, coxeter orbifolds and torus actions.
\newblock {\em Duke Mathematical Journal}, 62(2):417--451, 1991.

\bibitem[DKU19]{dasgupta2019cohomology}
Jyoti Dasgupta, Bivas Khan, and Vikraman Uma.
\newblock Cohomology of torus manifold bundles.
\newblock {\em Mathematica Slovaca}, 69(3):685--698, 2019.

\bibitem[GKM98]{goresky1998equivariant}
Mark Goresky, Robert Kottwitz, and Robert MacPherson.
\newblock Equivariant cohomology, koszul duality, and the localization theorem.
\newblock {\em Inventiones mathematicae}, 131(1):25--84, 1998.

\bibitem[Hat02]{Hatcher}
A.~Hatcher.
\newblock {\em Algebraic topology}.
\newblock Cambridge University Press, Cambridge, 2002.

\bibitem[HK17]{hasui2017p}
Sho Hasui and Daisuke Kishimoto.
\newblock p-local stable cohomological rigidity of quasitoric manifolds.
\newblock {\em Osaka Journal of Mathematics}, 54(2):343--350, 2017.

\bibitem[HKM20]{hof2020}
Johannes Hofscheier, Askold Khovanskii, and Leonid Monin.
\newblock Cohomology rings of toric bundles and the ring of conditions.
\newblock {\em arXiv preprint arXiv:2006.12043}, 2020.

\bibitem[HKS16]{hasui2016p}
Sho Hasui, Daisuke Kishimoto, and Takashi Sato.
\newblock p-local stable splitting of quasitoric manifolds.
\newblock {\em Osaka Journal of Mathematics}, 53(3):843--854, 2016.

\bibitem[HM03]{multifan}
Akio Hattori and Mikiya Masuda.
\newblock Theory of multi-fans.
\newblock {\em Osaka Journal of Mathematics}, 40(1):1--68, 2003.

\bibitem[Hof19]{roch}
J.~Hofscheier.
\newblock The ring of conditions for horospherical homogeneous spaces.
\newblock In {\em Proceedings of the conference Interactions with Lattice
  Polytopes Magdeburg}, 2019.

\bibitem[Kav11]{KavehVolume}
K.~Kaveh.
\newblock Note on cohomology rings of spherical varieties and volume
  polynomial.
\newblock {\em J. Lie Theory}, 21(2):263--283, 2011.

\bibitem[Kho21]{key}
Askold Khovanskii.
\newblock Generalized virtual polyhedra and cohomology of torus manifolds.
\newblock Talk at International School 'Toric Topology and Combinatorics',
  2021.

\bibitem[KKE21]{kazarnovskii2021newton}
Boris Kazarnovskii, Askold Khovanskii, and Alexander Esterov.
\newblock Newton polytopes and tropical geometry.
\newblock {\em Russian Mathematical Surveys}, 76(1):91, 2021.

\bibitem[KLM21]{quasimulti}
Askold Khovanskii, Ivan Limonchenko, and Leonid {Monin}.
\newblock Quasitoric manifolds and multi-polytopes.
\newblock {\em In preparation}, 2021.

\bibitem[KM21]{KhoM}
Askold Khovanskii and Leonid Monin.
\newblock Gorenstein algebras and toric bundles.
\newblock {\em arXiv preprint arXiv:2106.15562}, 2021.

\bibitem[Kou76]{Kouchnirenko}
A.~G. Kouchnirenko.
\newblock Poly\`edres de {N}ewton et nombres de {M}ilnor.
\newblock {\em Invent. Math.}, 32(1):1--31, 1976.

\bibitem[Kur09]{kuroki2009introduction}
Shintar{\^o} Kuroki.
\newblock Introduction to gkm theory.
\newblock {\em Science and Technology}, 2009(0063179):0063179, 2009.

\bibitem[LLP18]{Limonchenko2017CalabiYH}
Ivan~Yu. Limonchenko, Zhi L{\"u}, and Taras~E. Panov.
\newblock Calabi yau hypersurfaces and su-bordism.
\newblock {\em Proceedings of the Steklov Institute of Mathematics},
  302:270--278, 2018.

\bibitem[LP16]{lu2016toric}
Zhi L{\"u} and Taras Panov.
\newblock On toric generators in the unitary and special unitary bordism rings.
\newblock {\em Algebraic \& Geometric Topology}, 16(5):2865--2893, 2016.

\bibitem[LW16]{Lu2016EXAMPLESOQ}
Zhi L{\"u} and Wei Wang.
\newblock Examples of quasitoric manifolds as special unitary manifolds.
\newblock {\em Mathematical Research Letters}, 23:1453--1468, 2016.

\bibitem[Mas99]{Masuda1999UnitaryTM}
Mikiya Masuda.
\newblock Unitary toric manifolds, multi-fans and equivariant index.
\newblock {\em Tohoku Mathematical Journal}, 51:237--265, 1999.

\bibitem[MP06]{panovmasuda}
Mikiya Masuda and Taras Panov.
\newblock On the cohomology of torus manifolds.
\newblock {\em Osaka Journal of Mathematics}, 43(3):711--746, 2006.

\bibitem[PK92a]{PK92}
A.~V. Pukhlikov and A.~G. Khovanskii.
\newblock Finitely additive measures of virtual polyhedra.
\newblock {\em Algebra i Analiz}, 4(2):161--185, 1992.

\bibitem[PK92b]{KP}
A.~V. Pukhlikov and A.~G. Khovanski\u{\i}.
\newblock The {R}iemann-{R}och theorem for integrals and sums of
  quasipolynomials on virtual polytopes.
\newblock {\em Algebra i Analiz}, 4(4):188--216, 1992.

\bibitem[PU12]{PanUs}
Taras Panov and Yuri Ustinovsky.
\newblock Complex-analytic structures on moment-angle manifolds.
\newblock {\em Mosc. Math. J.}, 12(1):149--172, 216, 2012.

\bibitem[SU03]{US03}
P.~Sankaran and V.~Uma.
\newblock Cohomology of toric bundles.
\newblock {\em Comment. Math. Helv.}, 78(3):540--554, 2003.

\end{thebibliography}
%\Addresses
\end{document}